\documentclass{aptpub}
\usepackage{amsfonts}

\usepackage{amsmath}

\authornames{R. M. \L OCHOWSKI} 
\shorttitle{On pathwise uniform approximation} 
                                   
\newtheorem{fact}[thm]{Fact}
\theoremstyle{plain}

\begin{document}

\title{On pathwise uniform approximation \\
of processes with c\`{a}dl\`{a}g trajectories \\
by processes with minimal total variation}

\authorone[Warsaw School of Economics]{Rafa\L \ M. \L ochowski} 

\addressone{ul. Madali\'nskiego 6/8, 02-513 Warsaw Poland} 

\emailone{rlocho@sgh.waw.pl} 

\begin{abstract}
For a real c\`{a}dl\`{a}g function $f$ and positive constant $c$ we find
another c\`{a}dl\`{a}g function, which has the smallest total variation
possible among the functions uniformly approximating $f$ with accuracy $c/2$%
. The solution is expressed with the truncated variation, upward truncated
variation and downward truncated variation introduced in \cite
{Lochowski:2008} and \cite{Lochowski:2011}. They are analogs of Hahn-Jordan
decomposition of a c\`{a}dl\`{a}g function with finite total variation but
are always finite even if the total variation is infinite. We apply obtained results
to general stochastic processes with c\`{a}dl\`{a}g trajectories and in the
special case of Brownian motion with drift we apply them to obtain full
characterisation of its truncated variation by calculating its Laplace
transform. We also calculate covariance of upward and downward truncated
variations of Brownian motion with drift.
\end{abstract}




\section{Introduction}

Let $X=\left( X_{t}\right) _{t\in \left[ a;b\right] }$ be a real valued
stochastic process with c\`{a}dl\`{a}g trajectories. In general, the total
path variation of $X,$ defined as 
\begin{equation*}
TV\left( X,\left[ a;b\right] \right) =\sup_{n}\sup_{a\leq
t_{0}<t_{1}<...<t_{n}\leq b}\sum_{i=1}^{n}\left|
X_{t_{i}}-X_{t_{i-1}}\right| ,
\end{equation*}
may be (and in many most important cases is)\ a.s. infinite. However, in the
neighborhood of every c\`{a}dl\`{a}g path we may easily find a function,
total variation of which is finite.

Let $f$ be a c\`{a}dl\`{a}g function $f:\left[ a;b\right] \rightarrow 
\mathbb{R} $ and let$\ c>0.$ The natural question arises, what is the
smallest possible\ (or infimum of) total variation of functions from the
ball $\left\{ g:\left\| f-g\right\| _{\infty }\leq \tfrac{1}{2}c\right\} ,$
where $\left\| f-g\right\| _{\infty }:=\sup_{s\in \left[ a;b\right] }\left|
f\left( s\right) -g\left( s\right) \right| .$ The bound from below reads as 
\begin{equation*}
TV\left( g,\left[ a;b\right] \right) \geq TV^{c}\left( f,\left[ a;b\right]
\right) ,
\end{equation*}
where 
\begin{equation}
TV^{c}\left( f,\left[ a;b\right] \right) :=\sup_{n}\sup_{a\leq
t_{0}<t_{1}<...<t_{n}\leq b}\sum_{i=1}^{n}\max \left\{ \left| f\left(
t_{i}\right) -f\left( t_{i-1}\right) \right| -c,0\right\}  \label{tv:def}
\end{equation}
and follows immediately from the inequality 
\begin{equation*}
\left| g\left( t_{i}\right) -g\left( t_{i-1}\right) \right| \geq \max
\left\{ \left| f\left( t_{i}\right) -f\left( t_{i-1}\right) \right|
-c,0\right\} .
\end{equation*}

In this paper we will show that in fact we have equality 
\begin{equation}
\inf \left\{ TV\left( f+h,\left[ a;b\right] \right) :\left\| h\right\|
_{\infty }\leq \tfrac{1}{2}c\right\} =TV^{c}\left( f,\left[ a;b\right]
\right) .  \label{tv2c}
\end{equation}
Moreover, we will show that for any $c\leq \sup_{s,u\in \left[ a;b\right]
}\left| f\left( s\right) -f\left( u\right) \right| $ there exist unique
c\`{a}dl\`{a}g function $h^{c}:\left[ a;b\right] \rightarrow \mathbb{R}$
such that $\left\| h^{c}\right\| _{\infty }\leq \tfrac{1}{2}c$ and for any $%
s\in \left( a;b\right] $ 
\begin{equation*}
TV\left( f+h^{c},\left[ a;s\right] \right) =TV^{c}\left( f,\left[ a;s\right]
\right) .
\end{equation*}

\begin{rem}
Since we deal with c\`{a}dl\`{a}g functions, more natural setting of our problem would be the investigation of 
\begin{equation*}
\inf \left\{ TV\left( f+h,\left[ a;b\right] \right) : h \text{ - c\`{a}dl\`{a}g}, d_{D}(f,f+h)\leq \tfrac{1}{2}c\right\}, 
\end{equation*}
where $d_{D}$ denotes Skorohod metric. Since the total variation does not depend on the (continuous and strictly increasing) change of variable $t$ and the function $h^c$ minimizing $ TV\left( f+h,\left[ a;b\right] \right)$ appears to be a c\`{a}dl\`{a}g one, solutions of both problems coincide.
\end{rem}

The bound (\ref{tv:def}) is called truncated variation and, although it is
optimal, it is not even clear when it is finite. However, for a given $c>0 $
and $p\geq 1$ from the inequality $\max \left\{ \left| x\right| -c,0\right\}
\leq \left| x\right| ^{p}/c^{p-1}$ it immediately follows that it is finite
whenever $p-$variation of the function $f$ is finite. Moreover, for any
c\`{a}dl\`{a}g function $f,$ $TV^{c}\left( f,\left[ a;b\right] \right) $ is
finite and (as it will be proved in the following sections) it is a
continuous, convex and decreasing function of the parameter $c>0$ although
the limit 
\begin{equation*}
\lim_{c\downarrow 0}\ TV^{c}\left( f,\left[ a;b\right] \right) =TV\left( f,%
\left[ a;b\right] \right)
\end{equation*}
may be infinite.

Since the total variation depends only on the increments of the function, a
more natural setting of our problem would be the following.

For a c\`{a}dl\`{a}g function $f:\left[ a;b\right] \rightarrow \mathbb{R}$
and $c>0$ find 
\begin{equation*}
\inf \left\{ TV\left( f+h,\left[ a;b\right] \right) :\left\| h\right\|
_{osc}\leq c\right\} ,
\end{equation*}
where $\left\| h\right\| _{osc}:=\sup_{s,u\in \left[ a;b\right] }\left|
h\left( s\right) -h\left( u\right) \right| .$ Note that $\left\| .\right\|
_{osc}$ is a norm on the classes of bounded functions which differ by a
constant.

Solution to this problem will be the same as the solution to the preceding
problem, i.e. 
\begin{equation}
\inf \left\{ TV\left( f+h,\left[ a;b\right] \right) :\left\| h\right\|
_{osc}\leq c\right\} =TV^{c}\left( f,\left[ a;b\right] \right) .
\label{tvceq}
\end{equation}
Moreover, there exists an optimal representative $h^{0,c}:\left[ a;b\right]
\rightarrow \mathbb{R}$ of the class of functions $h$ for which the equality
(\ref{tvceq}) is attained and such that $h^{0,c}\left( a\right) =0.$ The
connection between $h^{c}$ of the previous problem and $h^{0,c}$ is such
that for any $s\in \left[ a;b\right] ,$ 
\begin{equation*}
h^{c}\left( s\right) =\alpha _{0}+h^{0,c}\left( s\right) ,
\end{equation*}
where $\alpha _{0}= - \inf_{s\in \left[ a;b\right] }h^{0,c}\left( s\right) - 
\frac{1}{2}\left\| h^{0,c}\right\| _{osc}$ is the unique real for which 
\begin{equation*}
\left\| \alpha _{0}+h^{0,c}\right\| _{\infty }=\inf_{\alpha \in \mathbb{R}%
}\left\| \alpha +h^{0,c}\right\| _{\infty }=\tfrac{1}{2}\left\|
h^{0,c}\right\| _{osc}.
\end{equation*}

Moreover, we will prove that $h^{0,c}$ is a c\`{a}dl\`{a}g function with
possible jumps only in the points where function $f$ has jumps and that it
may be represented in the following form 
\begin{equation*}
h^{0,c}\left( s\right) =f\left( a\right) +UTV^{c}\left( f;\left[ a;s\right]
\right) -DTV^{c}\left( f;\left[ a;s\right] \right) -f\left( s\right) ,
\end{equation*}
where 
\begin{equation}
UTV^{c}\left( f,\left[ a;b\right] \right) :=\sup_{n}\sup_{a\leq
t_{0}<t_{1}<...<t_{n}\leq b}\sum_{i=1}^{n}\max \left\{ f\left( t_{i}\right)
-f\left( t_{i-1}\right) -c,0\right\} ,  \label{utv:def}
\end{equation}
\begin{equation}
DTV^{c}\left( f,\left[ a;b\right] \right) :=\sup_{n}\sup_{a\leq
t_{0}<t_{1}<...<t_{n}\leq b}\sum_{i=1}^{n}\max \left\{ f\left(
t_{i-1}\right) -f\left( t_{i}\right) -c,0\right\}  \label{dtv:def}
\end{equation}
are called upward and downward truncated variations of the function $f$
respectively. We will also show that 
\begin{equation}
TV^{c}\left( f,\left[ a;b\right] \right) =UTV^{c}\left( f,\left[ a;b\right]
\right) +DTV^{c}\left( f,\left[ a;b\right] \right) .  \label{sumutvdtv}
\end{equation}

Properties of truncated variation and two other quantities related - upward
and downward truncated variations, are up to some degree known. In
particular, in \cite{Lochowski:2011} there were calculated Laplace
transforms of $UTV^{c}\left( W,\left[ 0;S\right] \right) $ and $%
DTV^{c}\left( W,\left[ 0;S\right] \right) $ for $W_t=B_t + \mu t$ being a
standard Brownian motion with drift $\mu$ and $\left[ 0;S\right] $ being a
random interval, length of which is exponentially distributed and
independent from the underlying Brownian motion. 
\begin{rem}
In \cite{Lochowski:2011} the functionals $UTV^{c}\left( X,\left[ a;b\right]
\right)$ and $DTV^{c}\left( X,\left[ a;b\right] \right) $ were defined with
slightly different formulas, but it is easy to see that both definitions coincide.
\end{rem}
This also gives the full characterisation of the distribution of $%
UTV^{c}\left( W,\left[ 0;T\right] \right)$ and $DTV^{c}\left( W,\left[ 0;T%
\right]\right)$ for deterministic time $T$. But only the results of this
paper allowed to give the full characterisation of the distribution of $%
TV^{c}\left(W,\left[ 0;T\right]\right)$ and dependence structure between $%
UTV^{c}\left( W,\left[ 0;T\right] \right)$ and $DTV^{c}\left( W,\left[ 0;T%
\right]\right).$

The results of this paper also give a new interpretation of the results
obtained in \cite{Lochowski:2010fk} which may be stated in that way: for any
random path of $W,$ $W_{t}=\mu t+B_{t},t\in \left[ a;b\right] ,$ the total
variation of any random function $f:\left[ a;b\right] \rightarrow \mathbb{R}$
which is uniformly close to $W_t$ is bounded from below by 
\begin{equation*}
\frac{b-a}{2\varepsilon }+\sqrt{\frac{b-a}{3}}N_{\varepsilon },
\end{equation*}
where $\varepsilon =\sup_{t\in \left[ a;b\right] }\left| f\left( t\right)
-W_{t}\right| $ and $N_{\varepsilon }$ tends in distribution to a random
variable with a standard normal distribution $\mathcal{N}\left( 0;1\right) $
as $\varepsilon \downarrow 0.$

The paper is organized as follows. In the next section we introduce some
necessary definitions and notation, present the construction of the
functions $h^{c}$ and $h^{0,c}$ of the first and the second problem and
establish the connection between $h^{0,c}$ and truncated variation, upward
truncated variation and downward truncated variation. In the second section
we also summarize some general properties of (upward, downward) truncated
variation. In the third section we apply obtained results to general
processes with c\`{a}dl\`{a}g trajectories. In the last section we deal with
the Laplace transform of truncated variation, its moments and covariance
between upward and downward truncated variations of Brownian motion with
drift.

\section{Truncated variation, upward truncated variation and downward
truncated variation of a c\`{a}dl\`{a}g function - their optimality and
other properties}

\subsection{Definitions and notation}

In this subsection we introduce definitions and notation which will be used
throughout the whole section.

Let $-\infty <a<b<+\infty $ and let $f:\left[ a;b\right] \rightarrow \mathbb{%
R}$ be a c\`{a}dl\`{a}g function. For $c>0$ we define two stopping times 
\begin{gather*}
T_{D}^{c}f=\inf \left\{ s\geq a:\sup_{t\in \left[ a;s\right] }f\left(
t\right) -f\left( s\right) \geq c\right\} , \\
T_{U}^{c}f=\inf \left\{ s\geq a:f\left( s\right) -\inf_{t\in \left[ a;s%
\right] }f\left( t\right) \geq c\right\} .
\end{gather*}

Assume that $T_{D}^{c}f\geq T_{U}^{c}f$ i.e. the first upward jump of
function $f$ of size $c$ appears before the first downward jump of the same
size $c$ or both times are infinite (there is no upward or downward jump of
size $c$). Note that in the case $T_{D}^{c}f<T_{U}^{c}f$ we may simply
consider function $-f.$ Now we define sequences $\left( T_{U,k}^{c}\right)
_{k=0}^{\infty },\left( T_{D,k}^{c}\right) _{k=-1}^{\infty },$ in the
following way: $T_{D,-1}^{c}=a,$ $T_{U,0}^{c}=T_{U}^{c}f$ and for $%
k=0,1,2,...$%
\begin{gather*}
T_{D,k}^{c}=\left\{ 
\begin{array}{lr}
\inf \left\{ s\in \left[ T_{U,k}^{c};b\right] :\sup_{t\in \left[
T_{U,k}^{c};s\right] }f\left( t\right) -f\left( s\right) \geq c\right\} & 
\text{ if }T_{U,k}^{c}<b, \\ 
\infty & \text{ if }T_{U,k}^{c}\geq b,
\end{array}
\right. \\
T_{U,k+1}^{c}=\left\{ 
\begin{array}{lr}
\inf \left\{ s\in \left[ T_{D,k}^{c};b\right] :f\left( s\right) -\inf_{t\in 
\left[ T_{D,k}^{c};s\right] }f\left( t\right) \geq c\right\} & \text{ if }%
T_{D,k}^{c}<b, \\ 
\infty & \text{ if }T_{D,k}^{c}\geq b.
\end{array}
\right. 
\end{gather*}
\begin{rem}
\label{finK} Note that there exists such $K<\infty $ that $%
T_{U,K}^{c}=\infty $ or $T_{D,K}^{c}=\infty .$ Otherwise we would obtain two
infinite sequences $\left( s_{k}\right) _{k=1}^{\infty },\left( S_{k}\right)
_{k=1}^{\infty }$ such that $a\leq s_{1}<S_{1}<s_{2}<S_{2}<...$ $\leq b$ and 
$f\left( S_{k}\right) -f\left( s_{k}\right) \geq \tfrac{1}{2}c.$ But this is
a contradiction, since $f$ is a c\`{a}dl\`{a}g function and $\left( f\left(
s_{k}\right) \right) _{k=1}^{\infty },\left( f\left( S_{k}\right) \right)
_{k=1}^{\infty }$ have a common limit.
\end{rem}
\begin{rem}
\label{Norvaisa}
There exists such $K<\infty $ that 
$T_{U,K}^{c}=\infty $ or $T_{D,K}^{c}=\infty $ also in a more general case - when 
we assume that $f$ is not necessary a c\`{a}dl\`{a}g but only a \emph{regulated} function (cf. \cite[Corollary 2.2]{DudleyNorvaisa:2011}).
\end{rem}
Now let us define two sequences of non-decreasing functions $m_{k}^{c}:\left[
T_{D,k-1}^{c};T_{U,k}^{c}\right) \cap \lbrack a;b]\rightarrow \mathbb{R}$
and $M_{k}^{c}:\left[ T_{U,k}^{c};T_{D,k}^{c}\right) \cap \lbrack
a;b]\rightarrow \mathbb{R}$ for such $k$ that $T_{D,k-1}^{c}<\infty $ and $%
T_{U,k}^{c}<\infty $\ respectively, with the formulas 
\begin{equation*}
m_{k}^{c}\left( s\right) =\inf_{t\in \left[ T_{D,k-1}^{c};s\right] }f\left(
t\right) ,^{{}}M_{k}^{c}\left( s\right) =\sup_{t\in \left[ T_{U,k}^{c};s%
\right] }f\left( t\right) .
\end{equation*}

Next we define two finite sequences of real numbers $\left( m_{k}^{c}\right) 
$ and $\left( M_{k}^{c}\right) ,$ for such $k$ that $T_{D,k-1}^{c}<\infty $
and $T_{U,k}^{c}<\infty $\ respectively, with the formulas 
\begin{eqnarray*}
m_{k}^{c} &=&m_{k}^{c}\left( T_{U,k}^{c}-\right) =\inf_{t\in \left[
T_{D,k-1}^{c};T_{U,k}^{c}\right) \cap [a;b] }f\left( t\right) , \\
M_{k}^{c} &=&M_{k}^{c}\left( T_{D,k}^{c}-\right) =\sup_{t\in \left[
T_{U,k}^{c};T_{D,k}^{c}\right)\cap [a;b] }f\left( t\right) .
\end{eqnarray*}

\subsection{Solution of the first problem}

\label{sol1problem}

Let $f:\left[ a;b\right] \rightarrow \mathbb{R}$ be a fixed c\`{a}dl\`{a}g
function. In this subsection we will solve the following problem: \emph{what
is the smallest possible\ (or infimum of) total variation of functions from
the ball $\left\{ g:\left\| f-g\right\| _{\infty }\leq \tfrac{1}{2}c\right\}
?$ }

In order to solve this problem we start with results concerning
c\`{a}dl\`{a}g functions. We apply the definitions of the previous
subsection to the function $f$ and assume that $T_{D}^{c}f\geq T_{U}^{c}f.$
Define the function $f^{c}:\left[ a;b\right] \rightarrow \mathbb{R}$ with
the formulas 
\begin{equation*}
f^{c}\left( s\right) =\left\{ 
\begin{array}{lr}
m_{0}^{c}+c/2 & \text{ if }s\in \left[ a;T_{U,0}^{c}\right) ; \\ 
M_{k}^{c}\left( s\right) -c/2 & \text{ if }s\in \left[
T_{U,k}^{c};T_{D,k}^{c}\right) ,k=0,1,2,...; \\ 
m_{k+1}^{c}\left( s\right) +c/2 & \text{ if }s\in \left[
T_{D,k}^{c};T_{U,k+1}^{c}\right) ,k=0,1,2,....
\end{array}
\right.
\end{equation*}

\begin{rem}
Note that due to Remark \ref{finK}, $b$ belongs to one of the intervals $%
\left[ T_{U,k}^{c};T_{D,k}^{c}\right) $ or $\left[ T_{D,k}^{c};T_{U,k+1}^{c}%
\right) $ for some $k=0,1,2,...$ and the function $f^{c}$ is defined for
every $s\in \lbrack a;b].$
\end{rem}

\begin{rem}
One may think about the function $f^{c}$ as of the most ''lazy'' function
possible, which changes its value only if it is necessary for the relation $%
\left\| f-f^{c}\right\| _{\infty }\leq c/2$ to hold. The choice of
its starting value will become clear in the sequel.
\end{rem}

\begin{rem}
In the case $T_{D}^{c}f<T_{U}^{c}f$ we may apply the definitions of the
previous subsection to the function $-f$ and simply define $f^{c}=-(-f)^{c}.$
Thus we will assume that the mapping $f\mapsto f^{c}$ is defined for any 
c\`{a}dl\`{a}g function. Similarly, in all the proofs of this section we
will assume $T_{D}^{c}f\leq T_{U}^{c}f,$ but all results of this section
(i.e. Lemma \ref{lem1}, Theorem \ref{thm1}, Corollary \ref{cor1}, Lemma \ref
{lem2}, Theorem \ref{thm2}, Corollary \ref{cor2} and Theorem \ref{THMM})
apply to any c\`{a}dl\`{a}g function $f.$ Obvious modifications are only
necessary in the definition of the stopping times $T_{U,k}^{c}$ and $%
T_{D,k}^{c}$ and then the functions $f_{U}^{c}$ and $f_{D}^{c}$ of Theorem 
\ref{thm1}.
\end{rem}

We have the following

\begin{lem}
\label{lem1} The function $f^{c}$ uniformly approximates the function $f$ with
accuracy $c/2$ and has finite total variation. Moreover $f^{c}$ is
a c\`{a}dl\`{a}g function and every point of the discontinuity of $f^{c}$ is also a point of discontinuity of the function $f.$
\end{lem}

\begin{proof}
Let us fix $s\in \left[ a;b\right].$ We have three possibilities.
\begin{itemize}
\item  $s\in \left[ a;T_{U,0}^{c}\right) .$ In this case, since $a\leq
s<T_{U}^{c}f \leq T_{D}^{c}f,$ 
\begin{equation*}
f\left( s\right) -f^{c}\left( s\right) =f\left( s\right) -\inf_{t\in \left[
a;T_{U,0}^{c}\right) }f\left( t\right) -c/2\in \left[ -c/2;c/2\right) .
\end{equation*}

\item  $s\in \left[ T_{U,k}^{c};T_{D,k}^{c}\right) ,$ for some $k=0,1,2,...
$ In this case $M_{k}^{c}\left( s\right) -f\left( s\right) $ belongs to the
interval $\left[ 0;c\right) ,$ hence 
\begin{equation*}
f\left( s\right) -f^{c}\left( s\right) =f\left( s\right) -M_{k}^{c}\left(
s\right) +c/2\in \left( -c/2;c/2\right] .
\end{equation*}

\item  $s\in \left[ T_{D,k}^{c};T_{U,k+1}^{c}\right) $ for some $%
k=0,1,2,...$ In this case $f\left( s\right) -m_{k+1}^{c}\left( s\right) $
belongs to the interval $\left[ 0,c\right) ,$ hence 
\begin{equation*}
f\left( s\right) -f^{c}\left( s\right) =f\left( s\right) -m_{k+1}^{c}\left(
s\right) -c/2 \in \left[ -c/2;c/2\right) .
\end{equation*}
\end{itemize}

Function $f^{c}$ has finite total variation since it is non-decreasing on
the intervals $\left[ T_{U,k}^{c};T_{D,k}^{c}\right) ,k=0,1,2,...$ and
non-increasing on the intervals $\left[ T_{D,k}^{c};T_{U,k+1}^{c}\right)
,k=0,1,2,...,$ and it has finite number of jumps between these
intervals.

For similar reason, function $f^{c}$ has left and right limits. To see that
it is right-continuous, let us fix $s\in \left[ a;b\right] $ and notice that
by definition of $f^{c},$ for $t \in \left( s; b \right]$ sufficiently close to $s,$ 
\begin{equation*}
f^{c}\left( t\right) =\inf_{u\in \left[ s;t\right] }f^c\left( u\right)  \text{
or }f^{c}\left( t\right) =\sup_{u\in \left[ s;t\right] }f^c\left( u\right),
\end{equation*}
and the assertion follows from the right-continuity of the function $f.$

Similar argument may be applied to prove that $f^{c}$ is continuous in every
point of continuity of $f$ except the points 
$T_{U,0}^{c},T_{D,0}^{c},T_{U,1}^{c},T_{D,1}^{c},...;$ but if $s=T_{D,i}^{c}$
and $f$ is continuous at the point $s$ then it means that $f\left(
T_{U,i}^{c}-\right) =f\left( T_{U,i}^{c}\right) =\inf_{t\in \left[
T_{D,i-1}^{c};T_{U,i}^{c}\right) }f\left( t\right) +c$ and 
\begin{equation*}
f^{c}\left( T_{U,i}^{c}-\right) =\inf_{t\in \left[ T_{D,i-1}^{c};T_{U,i}^{c}%
\right) }f\left( t\right) +c/2=f\left( T_{U,i}^{c}\right) - c/2 =f^{c}\left( T_{U,i}^{c}\right) .
\end{equation*}
Similar argument applies when $s=T_{D,i}^{c}.$

\end{proof}

Since $f^{c}$ is of finite total variation, we know that there exist such
two non-decreasing functions $f_{U}^{c}$ and $f_{D}^{c}:\left[ a;b\right]
\rightarrow \left[ 0;+\infty \right) $ that $f^{c}\left( t\right)
=f^{c}\left( a\right) +f_{U}^{c}\left( t\right) -f_{D}^{c}\left( t\right) .$

Let us examine the sign of the jumps of function $f^{c}$ between intervals $%
\left[ T_{U,k}^{c};T_{D,k}^{c}\right) $ and $\left[
T_{D,k}^{c};T_{U,k+1}^{c}\right) .$ Due to c\`{a}dl\`{a}g property we have 
\begin{eqnarray*}
f^{c}\left( T_{U,k}^{c}\right) -f^{c}\left( T_{U,k}^{c}-\right)
&=&f^{c}\left( T_{U,k}^{c}\right) -m_{k}^{c}-c \\
&=&f\left( T_{U,k}^{c}\right) -\inf_{t\in \left[ T_{D,k-1}^{c};T_{U,k}^{c}%
\right) }f\left( t\right) -c\geq 0, \\
f^{c}\left( T_{D,k}^{c}\right) -f^{c}\left( T_{D,k}^{c}-\right)
&=&f^{c}\left( T_{D,k}^{c}\right) -M_{k}^{c}+2c \\
&=&-\left\{ \sup_{t\in \left[ T_{U,k}^{c};T_{D,k}^{c}\right) }f\left(
t\right) -f\left( T_{D,k}^{c}\right) \right\} +c\leq 0.
\end{eqnarray*}
Hence we may set $f_{U}^{c}\left( s\right) =f_{D}^{c}\left( s\right) =0$ for 
$s\in \left[ a;T_{U,0}^{c}\right) ,$ 
\begin{equation*}
f_{U}^{c}\left( s\right) =\left\{ 
\begin{array}{lr}
\sum_{i=0}^{k-1}\left\{ M_{i}^{c}-m_{i}^{c}-c\right\} +M_{k}^{c}\left(
s\right) -m_{k}^{c}-c & \text{ if }s\in \left[ T_{U,k}^{c};T_{D,k}^{c}%
\right) ; \\ 
\sum_{i=0}^{k}\left\{ M_{i}^{c}-m_{i}^{c}-c\right\} & \text{ if }s\in \left[
T_{D,k}^{c};T_{U,k+1}^{c}\right)
\end{array}
\right.
\end{equation*}
and 
\begin{equation*}
f_{D}^{c}\left( s\right) =\left\{ 
\begin{array}{lr}
\sum_{i=0}^{k-1}\left\{ M_{i}^{c}-m_{i+1}^{c}-c\right\} & \text{ if }s\in 
\left[ T_{U,k}^{c};T_{D,k}^{c}\right) ; \\ 
\sum_{i=0}^{k-1}\left\{ M_{i}^{c}-m_{i+1}^{c}-c\right\}
+M_{k}^{c}-m_{k+1}^{c}\left( s\right) -c & \text{ if }s\in \left[
T_{D,k}^{c};T_{U,k+1}^{c}\right) .
\end{array}
\right.
\end{equation*}

Now we will prove the following

\begin{thm}
\label{thm1} If $g:\left[ a;b\right] \rightarrow \mathbb{R}$ uniformly
approximates $f$ with accuracy $c/2,$ has finite total variation
and $g_{U},g_{D}:\left[ a;b\right] \rightarrow \left[ 0;+\infty \right) $
are such two non-decreasing functions that $g\left( t\right) =g\left(
a\right) +g_{U}\left( t\right) -g_{D}\left( t\right) ,t\in \left[ a;b\right]
,$ then for any $s\in \left[ a;b\right] $%
\begin{equation}
g_{U}\left( s\right) \geq f_{U}^{c}\left( s\right) \text{ and }g_{D}\left(
s\right) \geq f_{D}^{c}\left( s\right) .  \label{thm3:eq}
\end{equation}
\end{thm}

\begin{proof}
Again, we consider three cases.

\begin{itemize}
\item  $s\in \left[ a;T_{U,0}^{c}\right) .$ In this case $g_{U}\left(
s\right) \geq 0=f_{U}^{c}\left( s\right) $ as well as $g_{D}\left( s\right)
\geq 0=f_{D}^{c}\left( s\right) $

\item  $s\in \left[ T_{U,k}^{c};T_{D,k}^{c}\right) ,$ for some $k=0,1,2,...
$ In this case, from the fact that $g$ uniformly approximates $f$ with
accuracy $c/2$ and from the fact that $g_{U},g_{D}$ are non-decreasing, for $i=0,1,2,...k-1$ we get 
\begin{align*}
& \sup_{s_{i}\in \left[ T_{U,i}^{c};T_{D,i}^{c}\right) }g_{U}\left(
s_{i}\right) -\inf_{s_{i}\in \left[ T_{D,i-1}^{c};T_{U,i}^{c}\right)
}g_{U}\left( s_{i}\right)  \\
& \geq \sup_{s_{i}\in \left[ T_{U,i}^{c};T_{D,i}^{c}\right) }\left(
g_{U}-g_{D}\right) \left( s_{i}\right) -\inf_{s_{i}\in \left[
T_{D,i-1}^{c};T_{U,i}^{c}\right) }\left( g_{U}-g_{D}\right) \left(
s_{i}\right)  \\
& =\sup_{s_{i}\in \left[ T_{U,i}^{c};T_{D,i}^{c}\right) }g\left(
s_{i}\right) -\inf_{s_{i}\in \left[ T_{D,i-1}^{c};T_{U,i}^{c}\right)
}g\left( s_{i}\right)  \\
& \geq \sup_{s_{i}\in \left[ T_{U,i}^{c};T_{D,i}^{c}\right) }\left\{
f\left( s_{i}\right) -c/2\right\} -\inf_{s_{i}\in \left[
T_{D,i-1}^{c};T_{U,i}^{c}\right) }\left\{ f\left( s_{i}\right) +c/2\right\} 
\\
& =M_{i}^{c}-m_{i}^{c}-c.
\end{align*}
Similarly 
\begin{align*}
& g_{U}\left( s\right) -\inf_{s_{k}\in \left[ T_{D,k-1}^{c};T_{U,k}^{c}%
\right) }g_{U}\left( s_{k}\right)  \\
& =\sup_{t\in \left[ T_{U,k}^{c};s\right] }g_{U}\left( t\right)
-\inf_{s_{k}\in \left[ T_{D,k-1}^{c};T_{U,k}^{c}\right) }g_{U}\left(
s_{k}\right)  \\
& \geq \sup_{t\in \left[ T_{U,k}^{c};s\right] }\left( g_{U}-g_{D}\right)
\left( t\right) -\inf_{s_{k}\in \left[ T_{D,k-1}^{c};T_{U,k}^{c}\right)
}\left( g_{U}-g_{D}\right) \left( s_{k}\right)  \\
& =\sup_{t\in \left[ T_{U,k}^{c};s\right] }g\left( t\right) -\inf_{s_{k}\in %
\left[ T_{D,k-1}^{c};T_{U,k}^{c}\right) }g\left( s_{k}\right)  \\
& \geq \sup_{t\in \left[ T_{U,k}^{c};s\right] }\left\{ f\left( t\right)
-c/2\right\} -\inf_{s_{k}\in \left[ T_{D,k-1}^{c};T_{U,k}^{c}\right)
}\left\{ f\left( s_{k}\right) +c/2\right\}  \\
& =M_{k}^{c}\left( s\right) -m_{k}^{c}-c.
\end{align*}
Summing up the above inequalities and using monotonicity of $g_{U}$\ we
finally get 
\begin{equation*}
g_{U}\left( s\right) \geq \sum_{i=0}^{k-1}\left\{
M_{i}^{c}-m_{i}^{c}-c\right\} +M_{k}^{c}\left( s\right)
-m_{k}^{c}-c=f_{U}^{c}\left( s\right) .
\end{equation*}

The proof of the corresponding inequality for $g_{D}$ follows similarly and
we get
\begin{equation*}
g_{D}\left( s\right) \geq \sum_{i=0}^{k-1}\left\{
M_{i}^{c}-m_{i+1}^{c}-c\right\} =f_{D}^{c}\left( s\right) .
\end{equation*}

\item  $s\in \left[ T_{D,k}^{c};T_{U,k+1}^{c}\right) $
The proof follows similarly as in the previous case.
\end{itemize}
\end{proof}

From Theorem \ref{thm1} we immediately get that the decomposition 
\begin{equation}
f^{c}\left( s\right) =f^{c}\left( a\right) +f_{U}^{c}\left( s\right)
-f_{D}^{c}\left( s\right)  \label{decomp}
\end{equation}
is minimal (cf. \cite{Revuz:1991kx}, page 5) thus the total variation of the
function $f^{c}$ on the interval $\left[ a;s\right] $ equals $%
f_{U}^{c}\left( s\right) +f_{D}^{c}\left( s\right) .$ 
\begin{rem}
\label{fUfDcadlag} From Lemma \ref{lem1} and the minimality of the decomposition
(\ref{decomp}) it follows that $f_{U}^{c}$ and $f_{U}^{c}$ are also c\`{a}dl\`{a}g 
functions and that every point of their discontinuity is also a point
of discontinuity of the function $f.$ Moreover, due to the minimality of the variation of the 
function $f^c,$ any jump of $f^c$ is no greater than the jump of the function $f.$
\end{rem}
We also have 
\begin{corollary}
\label{cor1} The function $f^{c}$ is optimal i.e. if $g:\left[ a;b\right]
\rightarrow \mathbb{R}$ is such that $\left\| f-g\right\| _{\infty }\leq 
c/2$ and has finite total variation, then for every $s\in \left[ a;b\right] $ 
\begin{equation*}
TV\left( g,\left[ a;s\right] \right) \geq TV\left( f^{c},\left[ a;s\right]
\right) .
\end{equation*}
Moreover, it is unique in such a sense that if for every $s\in \left[ a;b\right] $ 
the opposite inequality holds 
\begin{equation*}
TV\left( g,\left[ a;s\right] \right) \leq TV\left( f^{c},\left[ a;s\right]
\right) 
\end{equation*}
and $c\leq \sup_{s,u\in \lbrack a;b]}|f(s)-f(u)|$ then $g=f^{c}.$
\end{corollary}

\begin{proof}
Let $g_{U},g_{D}:\left[ a;b\right] \rightarrow \left[ 0;+\infty \right) $ be
two non-decreasing functions such that for $s\in \left[ a;b\right] ,$ 
$g\left( s\right) =g\left( a\right) +g_{U}\left( s\right) -g_{D}\left(
s\right) $ and $TV\left( g, \left[ a;s\right]\right)  =g_{U}\left( s\right)
+g_{D}\left( s\right) .$ 

The first assertion follows directly from Theorem \ref{thm1} and the fact that $TV\left( g, \left[ a;s\right]\right)  =g_{U}\left( s\right)
+g_{D}\left( s\right) .$ 

The opposite inequality holds for every $s\in \left[ a;b\right] $\ iff 
$g_{U}\left( s\right) =f_{U}^{c}\left( s\right) $ and $g_{D}\left( s\right)
=f_{D}^{c}\left( s\right) .$ Thus in such a case we get $g\left( s\right)
-f^{c}\left( s\right) =g\left( a\right) -f^{c}\left( a\right) $ and we have 
\begin{eqnarray}
c/2 &\geq &\inf_{s\in \left[ a;T_{U,0}^{c}\right) }\left\{ g\left( s\right)
-f\left( s\right) \right\} =\inf_{s\in \left[ a;T_{U,0}^{c}\right) }\left\{
g\left( a\right) -f^{c}\left( a\right) +f^{c}\left( s\right) -f\left(
s\right) \right\}   \notag \\
&=&g\left( a\right) -f^{c}\left( a\right) +c/2  \label{g1}
\end{eqnarray}
(notice that  $T_{U,0}^{c} \leq b$ since $c \leq \sup_{s,u \in [a;b]} | f(s) - f(u) |$ and $T_{U,0}^{c} \leq T_{D,0}^{c}$).
On the other hand we have 
\begin{eqnarray}
-c/2 &\leq &g\left( T_{U,0}^{c}\right) -f\left( T_{U,0}^{c}\right) =g\left(
a\right) -f^{c}\left( a\right) +f^{c}\left( T_{U,0}^{c}\right) -f\left(
T_{U,0}^{c}\right)   \notag \\
&=&g\left( a\right) -f^{c}\left( a\right) -c/2.  \label{g2}
\end{eqnarray}
From (\ref{g1}) and (\ref{g2}) we get $g\left( a\right) =f^{c}\left(
a\right) .$ This together with the equalities $g_{U}\left( s\right)
=f_{U}^{c}\left( s\right) $ and $g_{D}\left( s\right) =f_{D}^{c}\left(
s\right) $ gives $g=f^{c}.$
\end{proof}

The formula obtained for the smallest possible total variation of a function
from the ball $\left\{ g:\left\| f-g\right\| \leq c/2\right\} $ reads as 
\begin{equation*}
f_{U}^{c}\left( b\right) +f_{D}^{c}\left( b\right)
\end{equation*}
and does not resemble formula (\ref{tv2c}). In subsection \ref
{relationutvdtv} we will show that these formulas coincide.

\subsection{Solution of the second problem}

In this subsection we will solve the following problem: \emph{\ for a
c\`{a}dl\`{a}g function $f:\left[ a;b\right] \rightarrow \mathbb{R}$ and $c>0
$ find 
\begin{equation*}
\inf \left\{ TV\left( f+h,\left[ a;b\right] \right) :\left\| h\right\|
_{osc}\leq c\right\} ,
\end{equation*}
where $\left\| h\right\| _{osc}:=\sup_{s,u\in \left[ a;b\right] }\left|
h\left( s\right) -h\left( u\right) \right| .$}

We will show that 
\begin{equation*}
\inf \left\{ TV\left( f+h,\left[ a;b\right] \right) :\left\| h\right\|
_{osc}\leq c\right\} =f_{U}^{c}\left( b\right) +f_{D}^{c}\left( b\right) ,
\end{equation*}
where $f_{U}^{c}$ and $f_{D}^{c}$ were defined in the previous subsection.
In order to do it let us simply define 
\begin{equation*}
f^{i,c}=f_{U}^{c}-f_{D}^{c}.
\end{equation*}
We have

\begin{lem}
\label{lem2} The increments of the function $f^{i,c}$ uniformly approximate the
increments of the function $f$ with accuracy $c$ and the function $f^{i,c}$ has
finite total variation.
\end{lem}

\begin{proof}
Since the difference $f^c - f^{i,c}$ is constant, the first and the second assertion follows immediately from 
Lemma \ref{lem1} and from simple calculation that for any $s,u \in [a;b],$
\begin{eqnarray*}
\lefteqn{ \left\{ f^{i,c}\left( s\right) -f^{i,c}\left( u\right) \right\}- \left\{ f\left( s\right) -f\left( u\right) \right\} }\\
& = &\left\{ f^{c}\left( s\right) -f\left( s\right) \right\} -\left\{
f^{c}\left( u\right) -f\left( u\right) \right\} \in [-c;c].
\end{eqnarray*}
\end{proof}

Now we will prove the analog of Theorem \ref{thm1}.

\begin{thm}
\label{thm2} If the increments of the function $g:\left[ a;b\right] \rightarrow 
\mathbb{R}$ uniformly approximate the increments of the function $f$ with
accuracy $c,$ $g$ has finite total variation and $g_{U},g_{D}:\left[
a;b\right] \rightarrow \left[ 0;+\infty \right) $ are such two
non-decreasing functions that $g\left( t\right) =g\left( a\right)
+g_{U}\left( t\right) -g_{D}\left( t\right) ,t\in \left[ a;b\right] ,$ then
for any $s\in \left[ a;b\right] $%
\begin{equation*}
g_{U}\left( s\right) \geq f_{U}^{c}\left( s\right) \text{ and }g_{D}\left(
s\right) \geq f_{D}^{c}\left( s\right) .
\end{equation*}
\end{thm}

\begin{proof}
It is enough to see that for $h = g - f,$ $\left\| h\right\|
_{osc}\leq c,$ thus for $${\alpha} = - \frac{1}{2} \left\{ \inf_{s \in [a;b] }h(s) + \sup_{s \in [a;b] }h(s) \right\},$$ $\left\| {\alpha} + h\right\|
_{\infty}\leq \tfrac{1}{2}c,$ and the function $g_{{\alpha}} = {\alpha} + g$ belongs to the ball $\left\{ g:\left\| f-g\right\| _{\infty }\leq \tfrac{1}{2}c\right\}.$ Application of Theorem \ref{thm1} to the function $g_{{\alpha}}$ concludes the proof.

\end{proof}

Since the decomposition $f^{i,c}\left( s\right) =$ $f_{U}^{c}\left( s\right)
-f_{D}^{c}\left( s\right) $ is minimal and $f^{i,c}\left( a\right) = 0$ we
immediately obtain

\begin{corollary}
\label{cor2} The function $f^{i,c}$ is optimal i.e. if $g:\left[ a;b\right]
\rightarrow \mathbb{R}$ is such that 
\begin{equation*}
\sup_{a\leq u<s\leq b}\left| \left\{ g\left( s\right) -g\left( u\right)
\right\} -\left\{ f\left( s\right) -f\left( u\right) \right\} \right| \leq c
\end{equation*}
and $g$ has finite total variation, then for every $s\in \left[ a;b\right] $ 
\begin{equation*}
TV\left( g,\left[ a;s\right] \right) \geq TV\left( f^{i,c},\left[ a;s\right]
\right) .
\end{equation*}
Moreover, it is unique in such a sense that if $g\left( a\right) =0$ and\
for every $s\in \left[ a;b\right] $ the opposite inequality holds 
\begin{equation*}
TV\left( g,\left[ a;s\right] \right) \leq TV\left( f^{i,c},\left[ a;s\right]
\right) ,
\end{equation*}
then $g=f^{i,c}.$
\end{corollary}

From Corollary \ref{cor2} it immediately follows that 
\begin{equation*}
\inf \left\{ TV\left( f+h,\left[ a;b\right] \right) :\left\| h\right\|
_{osc}\leq c\right\} =f_{U}^{c}\left( b\right) +f_{D}^{c}\left( b\right).
\end{equation*}
Indeed, for any $h$ such that $\left\| h\right\| _{osc} \leq c$ we put $g =
f + h$ and if $g$ has finite total variation then it satisfies the
assumptions of Corollary \ref{cor2} and we get 
\begin{equation*}
TV\left( g,\left[ a;b\right] \right) \geq TV\left( f^{i,c},\left[ a;b\right]
\right) = f_{U}^{c}\left( b\right) +f_{D}^{c}\left( b\right) .
\end{equation*}

\subsection{Relation of the solutions of the first and the second problem
with truncated variation, upward truncated variation and downward truncated
variation}

\label{relationutvdtv}

In order to prove (\ref{tv2c}), (\ref{tvceq}) and (\ref{sumutvdtv}), where $%
UTV^{c}\left( f,\left[ a;s\right] \right) $ and $DTV^{c}\left( f,\left[ a;s%
\right] \right) $ are defined by (\ref{utv:def}) and (\ref{dtv:def})
respectively,\ it is enough to prove

\begin{thm}
\label{THMM} For a given c\`{a}dl\`{a}g function $f:\left[ a;b\right]
\rightarrow \mathbb{R}$ and for any $s\in \left( a;b\right] $\ the following
equalities hold 
\begin{gather}
UTV^{c}\left( f,\left[ a;s\right] \right) =f_{U}^{c}\left( s\right) ,
\label{UTVfU} \\
DTV^{c}\left( f,\left[ a;s\right] \right) =f_{D}^{c}\left( s\right) ,
\label{DTVfD} \\
TV^{c}\left( f,\left[ a;s\right] \right) =f_{U}^{c}\left( s\right)
+f_{D}^{c}\left( s\right) .  \label{eq3}
\end{gather}
\end{thm}

\begin{proof}
Examining (with obvious modifications) the proof of Lemma 3 from \cite{Lochowski:2011}, we see that it may be applied to the c\`{a}dl\`{a}g (but not necessarily continuous)
function $f$\ and we obtain 
\begin{equation}
UTV^{c}\left( f, \left[ a;s\right] \right) =\sup_{a\leq t<u\leq \left(
T_{D}^{c}f\right) \wedge s}\left( f\left( u\right) -f\left( t\right)
-c\right) _{+}+UTV^{c}\left( f, \left[ \left( T_{D}^{c}f\right) \wedge
s;s\right] \right) .  \label{utvf}
\end{equation}
Now, from the assumption $T_{D}^{c}f\geq T_{U}^{c}f$ we get 
$T_{D}^{c}f=T_{D,0}^{c}$ and we have that 
\begin{equation*}
\sup_{a\leq t<u\leq \left( T_{D}^{c}f\right) \wedge s}\left( f\left(
u\right) -f\left( t\right) -c\right) _{+}=\left\{ 
\begin{array}{lr}
0 & \text{ if }s\in \left[ a;T_{U,0}^{c}\right) ; \\ 
M_{0}^{c}\left( s\right) -m_{0}^{c}-c & \text{ if }s\in \left[
T_{U,0}^{c};T_{D,0}^{c}\right) ; \\ 
M_{0}^{c}-m_{0}^{c}-c & \text{ if }s\geq T_{D,0}^{c}.
\end{array}
\right. 
\end{equation*}
Iterating the equality (\ref{utvf}) we obtain 
\begin{eqnarray*}
UTV^{c}\left( f, \left[ a;s\right]  \right) &=&\left\{ 
\begin{array}{lr}
0 & \text{ if }s\in \left[ a;T_{U,0}^{c}\right) ; \\ 
\sum_{i=0}^{k-1}\left( M_{i}^{c}-m_{i}^{c}-c\right) +M_{k}^{c}\left(
s\right) -m_{k}^{c}-c & \text{ if }s\in \left[ T_{U,k}^{c};T_{D,k}^{c}\right) ;
\\ 
\sum_{i=0}^{k}\left( M_{i}^{c}-m_{i}^{c}-c\right) & \text{ if }s\in \left[
T_{D,k}^{c};T_{U,k+1}^{c}\right) 
\end{array}
\right.  \\
&=&f_{U}^{c}\left( s\right) .
\end{eqnarray*}

\begin{rem}
Let us define sequence of stopping times $\tilde{T}_{D,0}^{c}=0,$ and for $%
k=0,1,2,...$%
\begin{equation*}
\tilde{T}_{D,k+1}^{c}=\inf \left\{ s>\tilde{T}_{D,k}^{c}:\sup_{t\in \left[ 
\tilde{T}_{D,k}^{c};s\right] }f\left( t\right) -f\left( s\right) =c\right\} .
\end{equation*}
\ Let us fix $s_{0}$ and define $k_{0}=\max \left\{ k:\tilde{T}%
_{D,k}^{c}\leq s_{0}\right\} .$ The immediate consequence of (\ref{utvf}) is the equality 
\begin{equation*}
UTV^{c}\left( f,  \left[ a;s_{0}\right]\right) =\sum_{k=1}^{k_{0}-1}\sup_{%
\tilde{T}_{D,k}^{c}\leq s<u\leq \tilde{T}_{D,k+1}^{c}}\left( f\left(
u\right) -f\left( s\right) -c\right) _{+}+UTV^{c}\left( f, \left[ 
\tilde{T}_{D,k_{0}}^{c};s_{0}\right]\right)  
\end{equation*}
which looks different from $f_{U}^{c}\left( s_0\right) .$ But it
is easy to notice that for all $k\geq 1$ such that $\tilde{T}%
_{D,k+1}^{c}<T_{U,1}^{c}f$ the summand $\sup_{\tilde{T}_{D,k}^{c}\leq s<u\leq 
\tilde{T}_{D,k+1}^{c}}\left( f\left( u\right) -f\left( s\right) -c\right)
_{+}$ is equal zero. Thus in fact both quantities coincide.
\end{rem}

Identically we prove that $DTV^{c}\left( f\right) \left[ a;s\right]
=f_{D}^{c}\left( s\right) .$

Now, in order to prove the equality (\ref{eq3}) simply notice that 
$TV^{c}\left( f, \left[ a;s\right]\right)  \geq 0$\ and if $s\in \left[
T_{U,k}^{c};T_{D,k}^{c}\right) $ 
\begin{eqnarray*}
TV^{c}\left( f, \left[ a;s\right]\right)   &\geq &\sum_{i=0}^{k-1}\left(
M_{i}^{c}-m_{i}^{c}-c\right) +\sum_{i=0}^{k-1}\left(
M_{i}^{c}-m_{i+1}^{c}-c\right) +M_{k}^{c}\left( s\right) -m_{k}^{c}-c \\
&=&f_{U}^{c}\left( s\right) +f_{D}^{c}\left( s\right) .
\end{eqnarray*}
Analogously, if $s\in \left[ T_{D,k}^{c};T_{U,k+1}^{c}\right) $ 
\begin{eqnarray*}
TV^{c}\left( f, \left[ a;s\right]\right)   &\geq &\sum_{i=0}^{k-1}\left(
M_{i}^{c}-m_{i}^{c}-c\right) +\sum_{i=0}^{k-1}\left(
M_{i}^{c}-m_{i+1}^{c}-c\right) +M_{k}^{c} -m_{k+1}^{c}\left(
s\right) -c \\
&=&f_{U}^{c}\left( s\right) +f_{D}^{c}\left( s\right) .
\end{eqnarray*}
Hence for all $s\in \left[ a;b\right] $ 
\begin{equation*}
TV^{c}\left( f, \left[ a;s\right]\right)  \geq f_{U}^{c}\left( s\right)
+f_{D}^{c}\left( s\right) .
\end{equation*}

So 
\begin{equation*}
TV^{c}\left( f, \left[ a;s\right]\right)  \geq UTV^{c}\left( f, \left[
a;s\right]\right)  +DTV^{c}\left( f, \left[ a;s\right]\right)  .
\end{equation*}
Since the opposite inequality is obvious, we finally get (\ref{eq3}).
\end{proof}

Now we see that by Corollary \ref{cor1} and Corollary \ref{cor2} functions $%
h^{c}=f^{c}-f$ and $h^{0,c}=f(a)+f^{i,c}-f= f(a) + UTV^{c}\left( f,\left[ a;.%
\right] \right) -DTV^{c}\left( f,\left[ a;.\right] \right) -f$ are optimal
and such that for any $s\in \left( a;b\right] $%
\begin{eqnarray*}
\inf \left\{ TV\left( f+h,\left[ a;s\right] \right) :\left\| h\right\|
_{\infty }\leq \tfrac{1}{2}c\right\} &=&TV\left( f+h^{c},\left[ a;s\right]
\right) \\
&=&TV^{c}\left( f,\left[ a;s\right] \right),
\end{eqnarray*}
\begin{eqnarray*}
\inf \left\{ TV\left( f+h,\left[ a;s\right] \right) :\left\| h\right\|
_{osc}\leq c\right\} &=&TV\left( f+h^{0,c},\left[ a;s\right] \right) \\
&=&TV^{c}\left( f,\left[ a;s\right] \right) .
\end{eqnarray*}
Moreover, by Remark \ref{fUfDcadlag}, $h^{c}$ and $h^{0,c}$ are also
c\`{a}dl\`{a}g functions and every point of their discontinuity is also a
point of discontinuity of the function $f.$

\subsection{Further properties of truncated variation, upward truncated
variation and downward truncated variation}

In this subsection we summarize basic properties of the defined functionals.
We start with

\subsubsection{Algebraic properties.}

For any $c>0$ we have 
\begin{gather}
DTV^{c}\left( f,\left[ a;b\right] \right) =UTV^{c}\left( -f,\left[ a;b\right]
\right) ,  \label{A1} \\
TV^{c}\left( f,\left[ a;b\right] \right) =UTV^{c}\left( f,\left[ a;b\right]
\right) +DTV^{c}\left( f,\left[ a;b\right] \right) .  \label{A2}
\end{gather}
Property (\ref{A1}) follows simply from the definitions (\ref{utv:def}) and (%
\ref{dtv:def}). Property (\ref{A2}) is the consequence of Theorem \ref{THMM}.

\subsubsection{Properties of $UTV^{c}\left( f,\left[ a;b\right] \right)
,DTV^{c}\left( f,\left[ a;b\right] \right) $ and $TV^{c}\left( f,\left[ a;b%
\right] \right) $ as the functions of the parameter $c.$}

We have the following

\begin{fact}
For any c\`{a}dl\`{a}g function $f$ the functions $\left( 0;\infty \right)
\ni c\mapsto UTV^{c}\left( f,\left[ a;b\right] \right) \in \left[ 0;+\infty
\right) ,$ $\left( 0;\infty \right) \ni c\mapsto DTV^{c}\left( f,\left[ a;b%
\right] \right) \in \left[ 0;+\infty \right) $ and $\left( 0;\infty \right)
\ni c\mapsto TV^{c}\left( f,\left[ a;b\right] \right) \in \left[ 0;+\infty
\right) $ are nonincreasing, continuous, convex functions of the parameter $%
c.$ Moreover, $\lim_{c\downarrow 0}TV^{c}\left( f,\left[ a;b\right] \right)
=TV\left( f,\left[ a;b\right] \right) $ and for any $c\geq \left\| f\right\|
_{osc},$ $TV^{c}\left( f,\left[ a;b\right] \right) =0.$
\end{fact}

\begin{proof}
The finiteness of $TV,$ $UTV$ and $DTV$ follows from Lemma \ref{lem1} and Theorem \ref{THMM}. Monotonicity is obvious. 

We start with the proof of the convexity. Let us fix $c,\varepsilon >0$\ and
consider such a partition $a\leq t_{0}<t_{1}<...<t_{n}\leq b$ of the
interval $\left[ a;b\right] $ that 
\begin{equation*}
UTV^{c}\left( f,\left[ a;b\right] \right) \leq \sum_{i=0}^{n-1}\max \left\{
f\left( t_{i+1}\right) -f\left( t_{i}\right) -c,0\right\} +\varepsilon .
\end{equation*}
Taking $\alpha \in \left[ 0;1\right] $ and $c_{1},c_{2}>0$ such that $%
c=\alpha c_{1}+\left( 1-\alpha \right) c_{2}$\ we have the inequality 
\begin{multline*}
\max \left\{ f\left( t_{i+1}\right) -f\left( t_{i}\right) -\alpha
c_{1}-\left( 1-\alpha \right) c_{2},0\right\}  \\
=\max \left\{ \alpha \left( f\left( t_{i+1}\right) -f\left( t_{i}\right)
-c_{1}\right) +\left( 1-\alpha \right) \left( f\left( t_{i+1}\right)
-f\left( t_{i}\right) -c_{2}\right) ,0\right\}  \\
\leq \alpha \max \left\{ f\left( t_{i+1}\right) -f\left( t_{i}\right)
-c_{1},0\right\} +\left( 1-\alpha \right) \max \left\{ f\left(
t_{i+1}\right) -f\left( t_{i}\right) -c_{2},0\right\} .
\end{multline*}
Now 
\begin{eqnarray*}
UTV^{c}\left( f,\left[ a;b\right] \right)  &\leq &\sum_{i=0}^{n-1}\max
\left\{ f\left( t_{i+1}\right) -f\left( t_{i}\right) -c,0\right\}
+\varepsilon  \\
&\leq &\alpha \sum_{i=0}^{n-1}\max \left\{ f\left( t_{i+1}\right) -f\left(
t_{i}\right) -c_{1},0\right\}  \\
&&+\left( 1-\alpha \right) \sum_{i=0}^{n-1}\max \left\{ f\left(
t_{i+1}\right) -f\left( t_{i}\right) -c_{2},0\right\} +\varepsilon  \\
&\leq &\alpha UTV^{c_{1}}\left( f,\left[ a;b\right] \right) +\left( 1-\alpha
\right) UTV^{c_{2}}\left( f,\left[ a;b\right] \right) +\varepsilon .
\end{eqnarray*}
Since $\varepsilon $ may be arbitrary small, we obtain the convexity
assertion. From convexity and monotonicity we obtain the continuity
assertion.

The same properties of $DTV$ and $TV$ follow immediately from (\ref{A1}) and
(\ref{A2}).

The fact that for $c\geq \left\| f\right\| _{osc},$ $TV^{c}\left( f,\left[ a;b\right] \right) =0$ 
follows easily from equality 
$$\max \left\{ \left| f\left( t_{i+1}\right) -f\left( t_{i}\right) \right| -c,0\right\} =0 $$
satisfied for any such $c$ and $t_{i}, t_{i+1} \in [a;b].$
\end{proof}

\section{Application of the truncated variation to stochastic processes
with c\`{a}dl\`{a}g paths}

\subsection{Optimality of truncated variation processes}

Now we will apply the results of the previous section to a real-valued
stochastic process $\left(X_t\right)_{t \in [a;b]}$ with c\`{a}dl\`{a}g
paths.

By Theorems \ref{thm2} and \ref{THMM} we know that the increments of the
process $X^{i,c}$ defined for $s \in \left[a;b \right]$ as the difference 
\begin{equation}
X_{s}^{i,c}=UTV^{c}\left( X, \left[ a;s\right]\right)-DTV^{c}\left( X, \left[
a;s\right] \right),  \label{xicdef}
\end{equation}
uniformly approximates increments of the process $X$ with accuracy $c$ and
the process $X^{i,c}$ is the unique process starting from $0$ with the
smallest total variation possible on any interval $\left[ a;s\right],$ $s
\in \left(a;b \right],$ with such a property. Moreover, total variation $%
TV\left( X^{i,c}, \left[ a;s\right]\right)$ may be expressed as 
\begin{equation}  \label{xicvar}
TV\left( X^{i,c}, \left[ a;s\right]\right)= TV^{c}\left( X, \left[ a;s\right]
\right).
\end{equation}

It is important to note that stochastic process $X^{i,c}$ constructed with
formula (\ref{xicdef}) is \textbf{adapted} to the natural filtration of the
process $X.$ Since the process $X^{i,c}$ is adapted to the natural
filtration of the process $X,$ we also have the following

\begin{corollary}
\label{cor:xic}
Let $\left( X\right) _{t\in \left[ a;b\right] }$ be a process with c\`{a}dl%
\`{a}g trajectories, independent increments and such that for any $\varepsilon>0$
\begin{equation}
\mathbb{P}\left( \sup_{s,t\in \left[ a;b\right] }\left| X_{s}-X_{t}\right|
<\varepsilon \right) >0.  \label{cor:ineq}
\end{equation}
Assume that $\left( Y\right) _{t\in \left[ a;b\right]}$ 
is a stochastic process adapted to the natural filtration of $X,$ starting
from $0$ and such that the increments of $Y$ uniformly approximate
increments of $X$ with accuracy $c.$ We have $TV\left( Y, \left[ a;s\right]\right)  
\geq TV^{c}\left( X, \left[ a;s\right]\right)$ for any $s \in \left( a;b\right]$ and if the following relation holds a.s. 
\begin{equation}
TV\left( Y, \left[ a;b\right]\right)  \leq TV^{c}\left( X, \left[ a;b\right]\right) ,  \label{cor:adapt}
\end{equation}
then $Y=X^{i,c}$ a.s.
\end{corollary}

\begin{proof}
Let us assume that $TV\left( Y, \left[ a;b\right]\right) $ is finite and let $Y_U$ and $Y_D$ be two minimal non-decreasing processes such that for any $s\in \left[ a;b\right] ,$ 
$TV\left( Y, \left[ a;s\right] \right) =Y_U\left(
s\right) +Y_D\left( s\right) $ and $Y=Y_{U}-Y_D.$

By Theorems \ref{thm2} and \ref{THMM} we get 
\begin{equation*}
Y_U\geq UTV^{c}\left( X, \left[ a;s\right]\right) \text{ and }Y_D\geq DTV^{c}\left( X, \left[ a;s\right]\right).
\end{equation*}
If the equality $Y=X^{i,c}$ was not true a.s., then there would exist such $\varepsilon \in \left( 0;c\right) ,$ that with probability 
$p_{\varepsilon }>0$ for some $s_{0}\in \left[ a;b\right] $\ we had 
\begin{equation*}
TV\left( Y, \left[ a;s_{0}\right]\right)  >TV^{c}\left( X, \left[ a;s_{0}\right]\right) +\varepsilon .
\end{equation*}
Consider the event 
\begin{equation*}
A\left(s_{0}\right) =\left\{ \sup_{s,t\in \left[ s_{0};b\right] }\left| X_{s}-X_{t}\right|
<\varepsilon \right\} .
\end{equation*}
We have 
\begin{equation*}
\mathbb{P}\left( A\left(s_{0}\right)\right) \geq \mathbb{P}\left( \sup_{s,t\in \left[ a;b\right] }\left| X_{s}-X_{t}\right| <\varepsilon \right) =:q_{\varepsilon }>0.
\end{equation*}
By the construction of the process $X^{i,c}$ we see that for any $\omega
\in A\left(s_{0}\right)$ we have 
\begin{equation*}
TV\left( X^{i,c}, \left[ a;b\right]\right) = TV^{c}\left( X,  \left[
a;b\right]\right)  \leq TV^{c}\left( X,  \left[
a;s_{0}\right]\right) +\varepsilon ,
\end{equation*}
and if $ TV\left( Y, \left[ a;s_{0}\right]\right)  >TV^{c}\left( X, \left[ a;s_{0}\right]\right) +\varepsilon$ (which is independent from the event  $A\left(s_{0}\right)$), 
\begin{eqnarray*}
TV\left( Y,  \left[ a;b\right]\right)  &\geq &TV\left( Y, \left[ a;s_{0}\right] \right)  \\
&>&TV\left( X^{i,c}, \left[ a;s_{0}\right]\right)  +\varepsilon \geq TV\left(
X^{i,c}, \left[ a;b\right]\right).  
\end{eqnarray*}
Thus, from the independence of the increments of the process $X,$ the inequality 
$ TV\left( Y, \left[ a;b\right]\right)  \leq TV^{c}\left( X, \left[ a;b\right]\right)$ does not hold at least with the probability 
$p_{\varepsilon }\cdot q_{\varepsilon }>0.$
\end{proof}

We have already expressed the process $X^{i,c}$ with the elegant formula (%
\ref{xicdef}). Now, as we did in Subsection \ref{sol1problem} for a
c\`{a}dl\`{a}g function, we may state the problem of finding the process $%
X^c $ with the smallest variation possible, uniformly approximating paths of
the process $X$ with accuracy $c/2.$

By the results of the previous section we notice that we may easily express
the processes $X^{c}$ as 
\begin{equation*}
X^{c}=\alpha +UTV^{c}\left( X, \left[ a;s\right]\right) -DTV^{c}\left( X, %
\left[ a;s\right]\right),
\end{equation*}
where 
\begin{equation*}
\alpha =\left\{ 
\begin{array}{lr}
\inf_{a\leq t<\left(T_{U}^{c}X \right)\wedge b}X_{t}+ c/2 & \text{ if }%
T_{U}^{c}X\leq T_{D}^{c}X, \\ 
\sup_{a\leq t<T_{D}^{c}X}X_{t}-c/2 & \text{ if }T_{U}^{c}X>T_{D}^{c}X.
\end{array}
\right.
\end{equation*}
Notice that the total path variation of $X^c$ also reads as $TV^{c}\left( X, 
\left[ a;b\right]\right).$ Unfortunately, due to the definition of $\alpha
, $ $X^{c}$ may be not adapted (to the natural filtration of $X$) process.
This is the price for the minimality of the variation of $X^c$.

It is easy to see that adapted to the natural filtration and uniformly
approximating paths of $X$ - with accuracy $c$ - is the process 
\begin{equation*}
\tilde{X}^{i,c}_s = X_a + X^{i,c}_s,
\end{equation*}
total variation of which also reads as $TV^{c}\left( X, 
\left[ a;b\right]\right)$,  but here we pay the big price for adaptability - twice time smaller accuracy of
the approximation.

But it is easy to construct another process $\tilde{X}^{c}$, uniformly
approximating paths of $X$ with accuracy $c/2$ and adapted to the natural
filtration of the process $X,$ total variation of which does not differ much
from the total variation of $X^{c}.$ This process is defined in the
following way. Let 
\begin{eqnarray*}
T_{d}^{c} &=&\inf \left\{ s\geq a:X_{a}-\inf_{t\in \left[ a;s\right]
}X_{t}\geq c/2\right\} , \\
T_{u}^{c} &=&\inf \left\{ s\geq a:\sup_{t\in \left[ a;s\right]
}X_{t}-X_{a}\geq c/2\right\} .
\end{eqnarray*}
Assuming that $T_{u}^{c}<T_{d}^{c}$ (in the opposite case we simply consider
the process $-X$) we define the sequence of stopping times $\left(
T_{u,k}^{c}\right) _{k=0}^{\infty },\left( T_{d,k}^{c}\right)
_{k=-1}^{\infty },$ in the following way: $T_{d,-1}^{c}=a,$ $%
T_{u,0}^{c}=T_{u}^{c}$ and for $k=0,1,2,...$ 
\begin{gather*}
T_{d,k}^{c}=\left\{ 
\begin{array}{lr}
\inf \left\{ s\in \left[ T_{u,k}^{c};b\right] :\sup_{t\in \left[
T_{u,k}^{c};s\right] }X_{t}-X_{s}\geq c\right\} & \text{ if }T_{u,k}^{c}<b, \\ 
\infty & \text{ if }T_{u,k}^{c}\geq b,
\end{array}
\right.  \\
T_{u,k+1}^{c}=\left\{ 
\begin{array}{lr}
\inf \left\{ s\in \left[ T_{d,k}^{c};b\right] :X_{s}-\inf_{t\in \left[
T_{d,k}^{c};s\right] }X_{t}\geq c\right\} & \text{ if }T_{d,k}^{c}<b, \\ 
\infty & \text{ if }T_{d,k}^{c}\geq b.
\end{array}
\right. \text{ }
\end{gather*}
Now the process $\tilde{X}^{c}$ is defined in the following way 
\begin{equation*}
\tilde{X}_{s}^{c}=\left\{ 
\begin{array}{lr}
X_{a} & \text{ if }s\in \left[ a;T_{u,0}^{c}\right) ; \\ 
\sup_{t\in \left[ T_{u,k}^{c};s\right] }X_{t}-c/2 & \text{ if }s\in \left[
T_{u,k}^{c};T_{d,k}^{c}\right) ; \\ 
\inf_{t\in \left[ T_{d,k}^{c};s\right] }X_{t}+c/2 & \text{ if }s\in \left[
T_{d,k}^{c};T_{u,k+1}^{c}\right) .
\end{array}
\right. 
\end{equation*}
It is not difficult to see that $\Vert \tilde{X}^{c}-X\Vert _{\infty }\leq
c/2$ and 
\begin{equation*}
TV\left( \tilde{X}^{c},\left[ a;b\right] \right) \leq c/2+TV^{c}\left( X,%
\left[ a;b\right] \right).
\end{equation*}

\section{The Laplace transform of truncated variation process\ of Brownian
motion with drift stopped at exponential time}

In this section we will calculate the Laplace transform of\ truncated
variation process of standard Brownian motion with drift $W,$ i.e. the process $TV^{c}\left( W,s\right) := TV^{c}\left( W,[0;s]\right), s \geq 0,$   stopped at (independent from $W$) exponentially distributed time $S.$

\subsection{The Laplace transform} 

We begin with some auxiliary observations. Firstly let us notice that by
Theorem 2.3, on the set $\left\{ T_{D}^{c}W\geq T_{U}^{c}W 
\right\} ,$ applying definition of sequences $\left(T_{U,k}^{c}\right)_{k=0}^{\infty}$ and $\left(T_{D,k}^{c}\right)_{k=-1}^{\infty}$ (c.f. subsection 2.1) for the function $f=W,$
for $s\geq 0$ we obtain 
\begin{equation*}
TV^{c}\left( W,s\right) =\left\{ 
\begin{array}{lr}
0 & \text{ if }s\in \left[ 0;T_{U,0}^{c}\right) ; \\ 
\sum_{i=0}^{k-1}\left\{ M_{i}^{c}-m_{i}^{c}-c\right\}
+\sum_{i=0}^{k-1}\left\{ M_{i}^{c}-m_{i+1}^{c}-c\right\} \\ 
+M_{k}^{c}\left( T\right) -m_{k}^{c}-c &\text{ if }s\in \left[
T_{U,k}^{c};T_{D,k}^{c}\right) ; \\ 
\sum_{i=0}^{k}\left\{ M_{i}^{c}-m_{i}^{c}-c\right\} +\sum_{i=0}^{k-1}\left\{
M_{i}^{c}-m_{i+1}^{c}-c\right\} \\ 
+M_{k}^{c}-m_{k+1}^{c}\left( T\right) -c & \text{ if }s\in \left[
T_{D,k}^{c};T_{U,k+1}^{c}\right)
\end{array}
\right.
\end{equation*}
(although $T_{U,k}^{c},T_{D,k}^{c}$ were defined for a function with a
domain being the compact interval $\left[ a;b\right] ,$ the extension of
their definition to a function defined on a half-line is straightforward). By
the continuity of Brownian paths, on the set $\left\{ T_{D}^{c}W\geq
T_{U}^{c}W \right\} $ we have 
\begin{equation*}
W\left( T_{U,k}^{c}\right) =m_{k}^{c}+c,W\left( T_{D,k}^{c}\right)
=M_{k}^{c}-c,
\end{equation*}
hence 
\begin{equation*}
TV^{c}\left( W,s\right) =\left\{ 
\begin{array}{lr}
0 & \text{ if }s\in \left[ 0;T_{U,0}^{c}\right) ; \\ 
\sum_{i=0}^{k-1}\left\{ M_{i}^{c}-W\left( T_{U,i}^{c}\right) \right\}
+\sum_{i=0}^{k-1}\left\{ W\left( T_{D,i}^{c}\right) -m_{i+1}^{c}\right\} \\ 
+M_{k}^{c}\left( T\right) -W\left( T_{U,k}^{c}\right) & \text{ if }s\in \left[
T_{U,k}^{c};T_{D,k}^{c}\right) ; \\ 
\sum_{i=0}^{k}\left\{ M_{i}^{c}-W\left( T_{U,i}^{c}\right) \right\}
+\sum_{i=0}^{k-1}\left\{ W\left( T_{D,i}^{c}\right) -m_{i+1}^{c}\right\} \\ 
+W\left( T_{D,k}^{c}\right) -m_{k+1}^{c}\left( T\right) & \text{ if }s\in 
\left[ T_{D,k}^{c};T_{U,k+1}^{c}\right) .
\end{array}
\right.
\end{equation*}

Now for any $0\leq a\leq b<+\infty $ we define two auxiliary functions 
\begin{gather*}
U\left[ a;b\right] =\sup_{a\leq t\leq b}W_{t}-W_{a}, \\
D\left[ a;b\right] =W_{a}-\inf_{a\leq t\leq b}W_{t}
\end{gather*}
and for $s\geq 0$ define two quantities 
\begin{eqnarray*}
U^{c}\left( W,s\right) &=&\sum_{i=0}^{\infty }U\left[ T_{U,i}^{c}\wedge
s;T_{D,i}^{c}\wedge s\right] +\sum_{i=0}^{\infty }D\left[ T_{D,i}^{c}\wedge
s;T_{U,i+1}^{c}\wedge s\right] ; \\
D^{c}\left( W,s\right) &=&\sum_{i=0}^{\infty }D\left[ T_{D,i}^{c}\wedge
s;T_{U,i+1}^{c}\wedge s\right] +\sum_{i=0}^{\infty }U\left[
T_{D,i+1}^{c}\wedge s;T_{U,i+2}^{c}\wedge s\right] .
\end{eqnarray*}

Notice that on the set $\left\{ T_{D}^{c}W\geq T_{U}^{c}W
\right\} $ we have 
\begin{equation*}
TV^{c}\left( W,s\right) =U^{c}\left( W,s\right) .
\end{equation*}
Similarly, if $T_{D}^{c}W<T_{U}^{c}\left( W\right) ,$ then we apply
definitions of sequences $\left(T_{U,k}^{c}\right)_{k=0}^{\infty}$ and $\left(T_{D,k}^{c}\right)_{k=-1}^{\infty}$ for $f=-W$ and obtain 
\begin{equation*}
TV^{c}\left( W,s\right) =U^{c}\left( -W,s\right) .
\end{equation*}

Now let $S$ be an exponential random variable, independent from $W,$ with density $\nu e^{-\nu x}.$ By $M_{TV^{c}\left(
W,S\right) }\left( \lambda \right) $\ we denote moment generating function
of $TV^{c}\left( W,S\right) ,$ i.e. 
\begin{equation*}
M_{TV^{c}\left( W,S\right) }\left( \lambda \right) :=\mathbb{E}\left[ \exp \left(
\lambda \cdot TV^{c}\left( W,S\right) \right) \right] .
\end{equation*}
We have the following equation 
\begin{eqnarray}
\lefteqn{M_{TV^{c}\left( W,S\right) }\left( \lambda \right) } & \notag \\
\lefteqn{=\mathbb{E}\left[ \exp \left( \lambda \cdot U^{c}\left( W,S\right) \right) |S\geq
T_{U,0}^{c},T_{D}^{c}W\geq T_{U}^{c}W\right] \times P\left( S\geq
T_{U,0}^{c},T_{D}^{c}W\geq T_{U}^{c}W\right) } \notag \\
& +\mathbb{E}\left[ \exp \left( \lambda \cdot U^{c}\left( -W,S\right) \right) |S\geq
T_{U,0}^{c},T_{D}^{c}W<T_{U}^{c}W\right] \times P\left( S\geq
T_{U,0}^{c},T_{D}^{c}W<T_{U}^{c}W\right)  \notag \\
& +\mathbb{P}\left( \min \left\{ T_{U}^{c}W,T_{D}^{c}W\right\} >S\right) .
\label{MTV}
\end{eqnarray}
By the lack of memory of exponential distribution, strong Markov property and the
independence of the increments of Brownian motion we have
\begin{eqnarray}
\lefteqn{\mathbb{E}\left[ \exp \left( \lambda \cdot U^{c}\left( W,S\right) \right) |S\geq
T_{U,0}^{c},T_{D}^{c}W\geq T_{U}^{c}W\right]}  \notag \\
\lefteqn{=\mathbb{E}\exp \left( \lambda \cdot U^{c}\left( W,S+T_{U,0}^{c}\right) \right) 
= \mathbb{E} \left[ \exp \left( \lambda \cdot U^{c}\left( W,S+T_{U,0}^{c}\right) 
\right) ;S<T_{D,0}-T_{U,0}^{c}\right]  }\notag \\
&+\mathbb{E}\left[ \exp \left( \lambda \cdot U^{c}\left( W,S+T_{U,0}^{c}\right)
\right) ;S\geq T_{D,0}-T_{U,0}^{c}\right]  \notag \\
\lefteqn{=\mathbb{E}\left[ \exp \left( \lambda \cdot U\left[ T_{U,0}^{c};S+T_{U,0}^{c}\right]
\right) ;S<T_{D,0}^{c}-T_{U,0}^{c}\right] } \notag \\
&+\mathbb{E}\left[ \exp \left\{ \lambda \cdot U\left[ T_{U,0}^{c};T_{D,0}^{c}\right]
+\lambda \cdot D^{c}\left( W,S+T_{U,0}^{c}\right) \right\} ;S\geq
T_{D,0}^{c}-T_{U,0}^{c}\right]  \notag \\
\lefteqn{=\mathbb{E}\left[ \exp \left( \lambda \cdot U\left[ T_{U,0}^{c};S+T_{U,0}^{c}\right]
\right) ;S<T_{D,0}^{c}-T_{U,0}^{c}\right]   }  \notag \\
&+\mathbb{E}\left[ \exp \left( \lambda \cdot U\left[ T_{U,0}^{c};T_{D,0}^{c}\right]
\right) ;S\geq T_{D,0}^{c}-T_{U,0}^{c}\right] \mathbb{E}\exp \left( \lambda \cdot
D^{c}\left( W,S+T_{D,0}^{c}\right) \right). \label{expression1}
\end{eqnarray}
Notice that in all the calculations above, except the first line, the
starting value of $T_{U,0}^{c}\geq 0$ is irrelevant, we need only to know
the recursive definitions of $T_{D,0}^{c},T_{U,1}^{c},...;$ thus we may set 
$T_{U,0}^{c}=0$ and we have 
\begin{equation}
\mathbb{E}\left[ \exp \left( \lambda \cdot U^{c}\left( W,S+T_{U,0}^{c}\right)
\right) ;S<T_{D,0}^{c}-T_{U,0}^{c}\right]  =\mathbb{E}\left[ \exp \left( \lambda \cdot \sup_{0\leq t\leq S}W_{t}\right)
;S<T_{D}^{c}W\right] ,  \label{a}
\end{equation}
and 
\begin{equation}
\mathbb{E}\left[ \exp \left( \lambda \cdot U\left[ T_{U,0}^{c};T_{D,0}^{c}\right]
\right) ;S\geq T_{D,0}^{c}-T_{U,0}^{c}\right] = \mathbb{E} \left[ \exp \left( \lambda \cdot \sup_{0\leq t\leq
T_{D}^{c}W}W_{t}\right) ;S\geq T_{D}^{c}W\right].  \label{b}
\end{equation}
Similarly 
\begin{eqnarray}
\lefteqn{\mathbb{E}\exp \left( \lambda \cdot D^{c}\left( W,S+T_{D,0}^{c}\right) \right) }
\notag \\
\lefteqn{ = \mathbb{E}\left[ \exp \left( \lambda \cdot D^{c}\left( W,S+T_{D,0}^{c}\right)
\right) ;S<T_{U,1}^{c}-T_{D,0}^{c}\right] }  \notag \\
&+\mathbb{E}\exp \left( \lambda \cdot D\left[ T_{D,0}^{c};T_{U,1}^{c}\right] ;S\geq
T_{U,1}^{c}-T_{D,0}^{c}\right) \mathbb{E} \left[ \exp \left( \lambda \cdot U^{c}\left(
W,S+T_{U,1}^{c}\right) \right) \right]   \label{expression2}
\end{eqnarray}
and similarly we get 
\begin{equation}
\mathbb{E} \left[ \exp \left( \lambda \cdot D^{c}\left( W,S+T_{D,0}^{c}\right)
\right) ;S<T_{U,1}^{c}-T_{D,0}^{c}\right]  = \mathbb{E} \left[ \exp \left( -\lambda \cdot \inf_{0\leq t\leq S}W_{t}\right)
;S<T_{U}^{c}W\right]  \label{d}
\end{equation}
and 
\begin{equation}
\mathbb{E} \exp \left( \lambda \cdot D\left[ T_{D,0}^{c};T_{U,1}^{c}\right] \right) 
 = \mathbb{E} \left[ \exp \left( -\lambda \cdot \inf_{0\leq t\leq
T_{U}^{c}W}W_{t}\right) ;S\geq T_{U}^{c}W\right] .  \label{e}
\end{equation}

Now, substituting in (\ref{expression1}) expression (\ref{expression2}) for
$\mathbb{E} \exp \left( \lambda \cdot D^{c}\left( W,S+T_{D,0}^{c}\right) \right) ,$ and
using (\ref{a})-(\ref{b}) and (\ref{d})-(\ref{e})\ we get 
\begin{eqnarray}
\lefteqn{ \mathbb{E} \left[ \exp \left( \lambda \cdot U^{c}\left( W,S\right) \right) |S\geq
T_{U,0}^{c},T_{D}^{c}W\geq T_{U}^{c}W\right] } \notag \\
& = \frac{\mathbb{E} \left[ \exp \left( \lambda \cdot \sup_{0\leq t\leq S}W_{t}\right)
;S<T_{D}^{c}W\right] }
{1-\mathbb{E}
\left[ \exp \left( \lambda \cdot \sup_{0\leq t\leq T_{D}^{c}W}W_{t}\right)
;S\geq T_{D}^{c}W\right] \mathbb{E}\left[ \exp \left( -\lambda \cdot \inf_{0\leq
t\leq T_{U}^{c}W}W_{t}\right) ;S\geq T_{U}^{c}W\right] }  \notag  \\
&  +  \frac{\mathbb{E} \left[ \exp \left( \lambda \cdot \sup_{0\leq t\leq
T_{D}^{c}W}W_{t}\right) ;S\geq T_{D}^{c}W\right] \mathbb{E} \left[ \exp \left(
-\lambda \cdot \inf_{0\leq t\leq S}W_{t}\right) ;S<T_{U}^{c}W\right] }
{1-\mathbb{E}
\left[ \exp \left( \lambda \cdot \sup_{0\leq t\leq T_{D}^{c}W}W_{t}\right)
;S\geq T_{D}^{c}W\right] \mathbb{E}\left[ \exp \left( -\lambda \cdot \inf_{0\leq
t\leq T_{U}^{c}W}W_{t}\right) ;S\geq T_{U}^{c}W\right] }  \label{g}
\end{eqnarray}

Using results of \cite{Taylor:1975} we will be able to calculate all quantities appearing
in (\ref{g}).

To calculate $\mathbb{E} \left[ \exp \left( \lambda \cdot \sup_{0\leq t\leq
S}W_{t}\right) ;S<T_{D}^{c}W\right] $ we will use formulas appearing in \cite[page 236]{Taylor:1975} . Denote $\tau \left( x\right) =\inf \left\{ t\geq 0:\sup_{0\leq
s\leq t}W_{s}=x\right\} .$\ We have equality (note that in the notation of
\cite{Taylor:1975} $S$ is denoted by $\xi $ with parameter $\beta =\nu ,$ $T$ is denoted
by $T_{D}^{c}W$ and $c$ is denoted by $a$) 
\begin{eqnarray*}
\mathbb{P}\left( \sup_{0\leq t\leq S}W_{t}>x,S<T_{D}^{c}W\right) &=&\mathbb{P}\left( \tau
\left( x\right) <S<T_{D}^{c}W\right) =\mathbb{P} \left( \tau \left( x\right) <S\leq
T_{D}^{c}W\right) \\
&=&\mathbb{P}\left( \tau \left( x\right) \leq T_{D}^{c}W,\tau \left( x\right)
<S\right) -\mathbb{P}\left( \tau \left( x\right) \leq T_{D}^{c}W<S\right) \\
&=&\exp \left( -\theta _{\mu }\left( \nu \right) x\right) \left[ 1-\mathbb{E}\exp
\left( -\nu T_{D}^{c}W\right) \right] \\
&=&\exp \left( -\theta _{\mu }\left( \nu \right) x\right) \left[ 1-e^{-\mu
c}V_{\mu }\left( \nu \right) /\theta _{\mu }\left( \nu \right) \right] ,
\end{eqnarray*}
where we define 
\begin{equation*}
\theta _{\mu }\left( \nu \right) =\sqrt{\mu ^{2}+2\nu }\coth \left( c\sqrt{%
\mu ^{2}+2\nu }\right) -\mu
\end{equation*}
and 
\begin{equation*}
V_{\mu }\left( \nu \right) =\frac{\sqrt{\mu ^{2}+2\nu }}{\sinh \left( c\sqrt{%
\mu ^{2}+2\nu }\right) }\text{ }.
\end{equation*}
Thus, for $\lambda $ such that $\Re \left( \lambda \right) <\theta _{\mu
}\left( \nu \right) ,$ 
\begin{equation*}
\mathbb{E}\left[ \exp \left( \lambda \cdot \sup_{0\leq t\leq S}W_{t}\right)
;S<T_{D}^{c}W\right] =\frac{\theta _{\mu }\left( \nu \right) -e^{-\mu
c}V_{\mu }\left( \nu \right) }{\theta _{\mu }\left( \nu \right) -\lambda }.
\end{equation*}
Further, by definition of $T_{D}^{c}W$ we have $\sup_{0\leq t\leq
T_{D}^{c}W}W_{t}=W_{T_{D}^{c}W}+c.$\ By this and by the independence of $S$\
from $T_{D}^{c}W$ we calculate 
\begin{eqnarray*}
\mathbb{E}\left[ \exp \left( \lambda \cdot \sup_{0\leq t\leq T_{D}^{c}W}W_{t}\right)
;S\geq T_{D}^{c}W\right] &=&\mathbb{E}\left[ \exp \left( \lambda \cdot \sup_{0\leq
t\leq T_{D}^{c}W}W_{t}\right) \exp \left( -\nu T_{D}^{c}W\right) \right] \\
&=&\mathbb{E}\left[ \exp \left( \lambda \cdot \left( W_{T_{D}^{c}W}+c\right) -\nu
T_{D}^{c}W\right) \right] \\
&=&e^{\lambda c}\mathbb{E}\left[ \exp \left( \lambda \cdot W_{T_{D}^{c}W}-\nu
T_{D}^{c}W\right) \right] .
\end{eqnarray*}
Now, utilizing the main result of \cite{Taylor:1975} i.e. equation (1.1), we have 
\begin{equation*}
e^{\lambda c}\mathbb{E}\left[ \exp \left( \lambda \cdot W_{T_{D}^{c}W}-\nu
T_{D}^{c}W\right) \right] =\frac{e^{-\mu c}V_{\mu }\left( \nu \right) }{%
\theta _{\mu }\left( \nu \right) -\lambda }.
\end{equation*}
Similarly, using symmetry, for $\lambda $ such that $\Re \left( \lambda
\right) <\theta _{-\mu }\left( \nu \right) ,$ 
\begin{equation*}
\mathbb{E}\left[ \exp \left( -\lambda \cdot \inf_{0\leq t\leq S}W_{t}\right)
;S<T_{U}^{c}W\right] =\frac{\theta _{-\mu }\left( \nu \right) -e^{\mu
c}V_{\mu }\left( \nu \right) }{\theta _{-\mu }\left( \nu \right) -\lambda }
\end{equation*}
and 
\begin{equation*}
\mathbb{E}\left[ \exp \left( -\lambda \cdot \inf_{0\leq t\leq T_{U}^{c}W}W_{t}\right)
;S\geq T_{U}^{c}W\right] =\frac{e^{\mu c}V_{\mu }\left( \nu \right) }{\theta
_{-\mu }\left( \nu \right) -\lambda }.
\end{equation*}
Substituting the above formulas into (\ref{g}) and simplifying, for 
$\lambda $ such that $\Re \left( \lambda \right) <\min \left\{ \theta _{\mu
}\left( \nu \right) ,\theta _{-\mu }\left( \nu \right) \right\} $\ we obtain 
\begin{eqnarray}
\lefteqn{ \mathbb{E} \left[ \exp \left( \lambda \cdot U^{c}\left( W,S\right) \right) |S\geq
T_{U,0}^{c},T_{D}^{c}W\geq T_{U}^{c}W\right]}  \notag \\
&=\dfrac{\dfrac{\theta _{\mu }\left( \nu \right) -e^{-\mu c}V_{\mu }\left(
\nu \right) }{\theta _{\mu }\left( \nu \right) -\lambda }+\dfrac{e^{-\mu
c}V_{\mu }\left( \nu \right) }{\theta _{\mu }\left( \nu \right) -\lambda }%
\dfrac{\theta _{-\mu }\left( \nu \right) -e^{\mu c}V_{\mu }\left( \nu \right) 
}{\theta _{-\mu }\left( \nu \right) -\lambda }}{1-\dfrac{\theta _{\mu }\left(
\nu \right) -e^{-\mu c}V_{\mu }\left( \nu \right) }{\theta _{\mu }\left( \nu
\right) -\lambda }\dfrac{\theta _{-\mu }\left( \nu \right) -e^{\mu c}V_{\mu
}\left( \nu \right) }{\theta _{-\mu }\left( \nu \right) -\lambda }}  \notag
\\
&=1+\lambda \dfrac{\theta _{-\mu }\left( \nu \right) +e^{-\mu c}V_{\mu
}\left( \nu \right) -\lambda }{\lambda ^{2}+2\nu +2\lambda \mu -2\lambda
\theta _{-\mu }\left( \nu \right) }.  \label{h}
\end{eqnarray}
To obtain formula for $\mathbb{E}\left[ \exp \left( \lambda \cdot U^{c}\left(
-W,S\right) \right) |S\geq T_{U,0}^{c},T_{D}^{c}W<T_{U}^{c}W\right] $ we
need only to change $\mu $ into $-\mu $ in the formula for $ \mathbb{E}\left[ \exp
\left( \lambda \cdot U^{c}\left( W,S\right) \right) |S\geq
T_{U,0}^{c},T_{D}^{c}W\geq T_{U}^{c}W\right] .$

To calculate probabilities appearing in the expression\ (\ref{MTV})\ for 
$M_{TV^{c}\left( W,S\right) }\left( \lambda \right) ,$ i.e. 
\begin{equation*}
\mathbb{P}\left( S\geq T_{U,0}^{c},T_{D}^{c}W\geq T_{U}^{c}W\right) =\mathbb{P}\left( S\geq
T_{U}^{c}W,T_{D}^{c}W>T_{U}^{c}W\right)
\end{equation*}
and 
\begin{equation*}
\mathbb{P}\left( S\geq T_{U,0}^{c},T_{D}^{c}W<T_{U}^{c}W\right) =\mathbb{P}\left( S\geq
T_{D}^{c}W,T_{D}^{c}W<T_{U}^{c}W\right)
\end{equation*}
we will use results of \cite{Lochowski:2011}. Since $S$ is independent from $\left(
T_{D}^{c}W,T_{U}^{c}W\right) ,$ we have 
\begin{eqnarray*}
\mathbb{P}\left( S\geq T_{U}^{c}W,T_{D}^{c}W>T_{U}^{c}W\right) &=&\mathbb{E}e^{-\nu
T_{U}^{c}W}I_{\left\{ T_{D}^{c}W>T_{U}^{c}W\right\} }, \\
\mathbb{P}\left( S\geq T_{D}^{c}W,T_{D}^{c}W<T_{U}^{c}W\right) &=&\mathbb{E}e^{-\nu
T_{D}^{c}W}I_{\left\{ T_{U}^{c}W>T_{D}^{c}W\right\} }.
\end{eqnarray*}
Using formula just below formula 19 in \cite{Lochowski:2011}, with $y=0,$ we get 
\begin{equation*}
\mathbb{E} e^{-\nu T_{D}^{c}W}I_{\left\{ T_{U}^{c}W>T_{D}^{c}W\right\} }=\left(
1-L_{0}^{-W}\left( \nu ;c\right) \right) \mathbb{E} e^{-\nu T_{D}^{c}W},
\end{equation*}
where (cf.  \cite[last but one formula on the page 389]{Lochowski:2011}) we have 
\begin{eqnarray*}
L_{0}^{-W}\left( \nu ;c\right) &=&\frac{\sqrt{\mu ^{2}+2\nu }}{2\nu }\left\{ 
\frac{e^{\mu c}\theta _{\mu }\left( \nu \right) }{\sinh \left( c\sqrt{\mu
^{2}+2\nu }\right) }-\frac{\sqrt{\mu ^{2}+2\nu }}{\sinh \left( c\sqrt{\mu
^{2}+2\nu }\right) ^{2}}\right\} \\
&=&\frac{V_{\mu }\left( \nu \right) }{2\nu }\left( e^{\mu c}\theta _{\mu
}\left( \nu \right) -V_{\mu }\left( \nu \right) \right) .
\end{eqnarray*}
Thus 
\begin{equation*}
\mathbb{P}\left( S\geq T_{D}^{c}W,T_{D}^{c}W<T_{U}^{c}W\right) =\left( 1-\frac{V_{\mu
}\left( \nu \right) }{2\nu }\left( e^{\mu c}\theta _{\mu }\left( \nu \right)
-V_{\mu }\left( \nu \right) \right) \right) \frac{e^{-\mu c}V_{\mu }\left(
\nu \right) }{\theta _{\mu }\left( \nu \right) }
\end{equation*}
and similarly 
\begin{equation*}
\mathbb{P}\left( S\geq T_{U}^{c}W,T_{D}^{c}W>T_{U}^{c}W\right) =\left( 1-\frac{V_{\mu
}\left( \nu \right) }{2\nu }\left( e^{-\mu c}\theta _{-\mu }\left( \nu
\right) -V_{\mu }\left( \nu \right) \right) \right) \frac{e^{\mu c}V_{\mu
}\left( \nu \right) }{\theta _{-\mu }\left( \nu \right) }.
\end{equation*}
Now, by (\ref{MTV}), (\ref{h}) and calculations above, we have
\begin{eqnarray*}
& \lefteqn{ M_{TV^{c}\left( W,S\right) }\left( \lambda \right) } \\
&=&\left( 1+\lambda \dfrac{\theta _{-\mu }\left( \nu \right) +e^{-\mu
c}V_{\mu }\left( \nu \right) -\lambda }{\lambda ^{2}+2\nu +2\lambda \mu
-2\lambda \theta _{-\mu }\left( \nu \right) }\right) P\left( S\geq
T_{U}^{c}W,T_{D}^{c}W>T_{U}^{c}W\right) \\
&&+\left( 1+\lambda \dfrac{\theta _{\mu }\left( \nu \right) +e^{\mu c}V_{\mu
}\left( \nu \right) -\lambda }{\lambda ^{2}+2\nu -2\lambda \mu -2\lambda
\theta _{\mu }\left( \nu \right) }\right) P\left( S\geq
T_{D}^{c}W,T_{D}^{c}W<T_{U}^{c}W\right) \\
& &  + 1-\mathbb{P}\left( S\geq T_{U}^{c}W,T_{D}^{c}W>T_{U}^{c}W\right) -P\left( S\geq
T_{D}^{c}W,T_{D}^{c}W<T_{U}^{c}W\right) \\
&= &1+\lambda \dfrac{\theta _{-\mu }\left( \nu \right) +e^{-\mu c}V_{\mu
}\left( \nu \right) -\lambda }{\lambda ^{2}+2\nu +2\lambda \mu -2\lambda
\theta _{-\mu }\left( \nu \right) }\left( e^{\mu c}-\dfrac{V_{\mu }\left( \nu
\right) }{2\nu }\theta _{-\mu }\left( \nu \right) +e^{\mu c}\dfrac{V_{\mu
}\left( \nu \right) ^{2}}{2\nu }\right) \dfrac{V_{\mu }\left( \nu \right) }{%
\theta _{-\mu }\left( \nu \right) } \\
&&+\lambda \dfrac{\theta _{\mu }\left( \nu \right) +e^{\mu c}V_{\mu }\left(
\nu \right) -\lambda }{\lambda ^{2}+2\nu -2\lambda \mu -2\lambda \theta
_{\mu }\left( \nu \right) }\left( e^{-\mu c}-\dfrac{V_{\mu }\left( \nu
\right) }{2\nu }\theta _{\mu }\left( \nu \right) +e^{-\mu c}\dfrac{V_{\mu
}\left( \nu \right) ^{2}}{2\nu }\right) \dfrac{V_{\mu }\left( \nu \right) }{%
\theta _{\mu }\left( \nu \right) }.
\end{eqnarray*}
Thus we have obtained

\begin{thm}
Let $W$ be a standard Wiener process with drift $\mu $ and $S$ be an
exponential random variable with density $\nu e^{-\nu x}I_{\left\{ x\geq
0\right\} },$ independent from $W.$ For any complex $\lambda $ such that $\Re \left( \lambda \right) <\min
\left\{ \theta _{\mu }\left( \nu \right) ,\theta _{-\mu }\left( \nu \right)
\right\} $ one has 
\begin{eqnarray}
& \lefteqn{\mathbb{E}\left[ \exp \left(
\lambda \cdot TV^{c}\left( W,S\right) \right) \right] } \label{MTV1} \\
&=&1+\lambda \frac{\theta _{-\mu }\left( \nu \right) +e^{-\mu c}V_{\mu
}\left( \nu \right) -\lambda }{\lambda ^{2}+2\nu +2\lambda \mu -2\lambda
\theta _{-\mu }\left( \nu \right) }\left( e^{\mu c}-\frac{V_{\mu }\left( \nu
\right) }{2\nu }\theta _{-\mu }\left( \nu \right) +e^{\mu c}\frac{V_{\mu
}\left( \nu \right) ^{2}}{2\nu }\right) \frac{V_{\mu }\left( \nu \right) }{%
\theta _{-\mu }\left( \nu \right) }  \notag \\
&&+\lambda \frac{\theta _{\mu }\left( \nu \right) +e^{\mu c}V_{\mu }\left(
\nu \right) -\lambda }{\lambda ^{2}+2\nu -2\lambda \mu -2\lambda \theta
_{\mu }\left( \nu \right) }\left( e^{-\mu c}-\frac{V_{\mu }\left( \nu
\right) }{2\nu }\theta _{\mu }\left( \nu \right) +e^{-\mu c}\frac{V_{\mu
}\left( \nu \right) ^{2}}{2\nu }\right) \frac{V_{\mu }\left( \nu \right) }{%
\theta _{\mu }\left( \nu \right) }.  \notag
\end{eqnarray}
\end{thm}

\subsection{Examples of applications of the formula (\ref{MTV1})}

\subsubsection{The first and the second moment of truncated variation process of Brownian motion
with drift stopped at exponential time}

Differentiating formula (\ref{MTV1}) we obtain 
\begin{equation}
\mathbb{E}TV^{c}\left( W,S\right) = \left[ \frac{\partial }{\partial \lambda }%
M_{TV^{c}\left( W,S\right) }\left( \lambda \right) \right] _{\lambda =0} 
=\frac{V_{\mu }\left( \nu \right) }{\nu }\cosh \left( c\mu \right)
\label{ETV}
\end{equation}
which agrees with the relation 
\begin{equation}
TV^{\mu }\left( W,S\right) =UTV^{c }\left( W,S\right) +DTV^{c }\left(
W,S\right)  \label{rel1}
\end{equation}
and already obtained in \cite{Lochowski:2011} formulas 
\begin{equation}
\mathbb{E}UTV^{c }\left( W,S\right) =\frac{e^{\mu c}V_{\mu }\left( \nu \right) }{%
2\nu },\mathbb{E}DTV^{c }\left( W,S\right) =\frac{e^{-\mu c}V_{\mu }\left( \nu
\right) }{2\nu }.  \label{rel3}
\end{equation}

Similarly, we calculate 
\begin{eqnarray}
\mathbb{E}TV^{c}\left( W,S\right) ^{2} &=&\left[ \frac{\partial ^{2}}{\partial
\lambda ^{2}}M_{TV^{c}\left( W,S\right) }\left( \lambda \right) \right]
_{\lambda =0}  \notag \\
&=&\frac{V_{\mu }\left( \nu \right) }{\nu ^{2}}\left( V_{\mu }\left( \nu
\right) +\cosh \left( \mu c\right) \theta _{\mu }\left( \nu \right) +e^{\mu
c}\mu \right) .  \label{rel2}
\end{eqnarray}
\begin{rem}
Inverting formulas (\ref{ETV}) and (\ref{rel2}), similarly as it was
done with formula (\ref{rel3}) in \cite[subsection 4.1]{Lochowski:2011} we may calulate
the first and the second moment of $TV^{c}\left( W,T\right) ,$ where $T$ is
a deterministic time. To invert the formula (\ref{rel2}) one may use
formulas from \cite[page 642]{BorodinSalminen:2002}.
\end{rem}
\subsubsection{Covariance of upward and downward truncated variation processes of Brownian motion with
drift stopped at exponential time}

Using (\ref{rel1}), (\ref{rel2}) and (\ref{rel3}) as well as results of
\cite[subsection 4.3]{Lochowski:2011} we are able to calculate the covariance of $UTV^{\mu
}\left( W,S\right) $ and $DTV^{c }\left( W,S\right) .$ Indeed, we have 
\begin{eqnarray}
&\lefteqn{ \mathbb{E}\left( UTV^{c }\left( W,S\right) \cdot DTV^{c }\left( W,S\right) %
\right) }  \notag \\
&=&\frac{1}{2}\left( \mathbb{E}TV^{c}\left( W,S\right) ^{2}-\mathbb{E}UTV^{c}\left( W,S\right)
^{2}-\mathbb{E}DTV^{c}\left( W,S\right) ^{2}\right)   \notag \\
&=&\frac{V_{\mu }\left( \nu \right) ^{2}}{2\nu ^{2}},  \label{EUTVDTV}
\end{eqnarray}
where we have used (\ref{rel2}), the folowing formula (cf. \cite[subsection 4.3]{Lochowski:2011}) 
\begin{eqnarray}
\mathbb{E}UTV^{c}\left( W,S\right) ^{2} &=&\int_{0}^{\infty }\mathbb{E}UTV^{c}\left(
W,t\right) ^{2}P\left( S\in dt\right)   \notag \\
&=&\nu \int_{0}^{\infty }e^{-\nu t}EUTV^{c}\left( W,t\right) ^{2}dt  \notag
\\
&=&\frac{e^{\mu c}V_{\mu }\left( \nu \right) \left( \mu ^{2}+2\nu -\nu
\left( 1-\cosh \left( 2c\sqrt{\mu ^{2}+2\nu }\right) \right) \right) }{2\nu
^{2}\theta _{\mu }\left( \nu \right) \sinh ^{2}\left( c\sqrt{\mu ^{2}+2\nu }%
\right) }  \notag \\
&=&\frac{e^{\mu c}V_{\mu }\left( \nu \right) \theta _{-\mu }\left( \nu
\right) }{2\nu ^{2}}  \label{EUTV2}
\end{eqnarray}
and the symmetric formula 
\begin{equation*}
\mathbb{E}DTV^{c}\left( W,S\right) ^{2}=\frac{e^{-\mu c}V_{\mu }\left( \nu \right)
\theta _{\mu }\left( \nu \right) }{2\nu ^{2}}.
\end{equation*}
Now we have 
\begin{eqnarray*}
&\lefteqn{ \text{Cov} \left( UTV^{c }\left( W,S\right) ,DTV^{c }\left( W,S\right) \right) }
\\
&=&\mathbb{E}\left( UTV^{c }\left( W,S\right) \cdot DTV^{c }\left( W,S\right) %
\right) -EUTV^{c }\left( W,S\right) \cdot EDTV^{c }\left( W,S\right)  \\
&=&\frac{V_{\mu }\left( \nu \right) ^{2}}{2\nu ^{2}}-\frac{e^{\mu c}V_{\mu
}\left( \nu \right) }{2\nu }\frac{e^{-\mu c}V_{\mu }\left( \nu \right) }{%
2\nu } \\
&=&\frac{V_{\mu }\left( \nu \right) ^{2}}{4\nu ^{2}} 
 =\frac{\mu ^{2}+2\nu }{4\nu ^{2}\left( \sinh \left( c\sqrt{\mu ^{2}+2\nu }\right) \right) ^{2}}>0.
\end{eqnarray*}
Thus we can observe that the correlation between $UTV^{c }\left(
W,S\right) $ and $DTV^{c }\left( W,S\right) $ is positive. This is due to the
fact that the magnitude of $UTV^{c }\left( W,S\right) $ and $DTV^{\mu
}\left( W,S\right) $ is highly dependent on the value of $S.$

\subsubsection{Covariance of UTV and DTV of Brownian motion
with drift}

Performing similar calculations to those in \cite[subsection 4.1]{Lochowski:2011} we may
simply obtain formulas for covariance between $UTV^{c }\left( W,T\right) $ and $DTV^{\mu
}\left( W,T\right) ,$ where $T$ is deterministic. Indeed, denoting by 
$\mathcal{L}_{\nu }^{-1}\left( g\right) $ the inverse of the Laplace
transform of the function $g\left( \nu \right) =\int_{0}^{\infty }e^{-\nu
t}f\left( t\right) dt,$ i.e. the function $f\left( t\right) ,$ we get 
\begin{equation*}
\mathcal{L}_{\nu }^{-1}\left( \nu ^{-3}\right) =t^{2}/2,
\end{equation*}
\begin{eqnarray*}
\mathcal{L}_{\nu }^{-1}\left( \frac{2\nu +\mu ^{2}}{\sinh ^{2}\left( c\sqrt{%
2\nu +\mu ^{2}}\right) }\right)  &=&\mathcal{L}_{\nu }^{-1}\left( \frac{%
2\left( \nu +\mu ^{2}/2\right) }{\sinh ^{2}\left( c\sqrt{2\left( \nu +\mu
^{2}/2\right) }\right) }\right)  \\
&=&e^{-\mu ^{2}t/2}\mathcal{L}_{\nu }^{-1}\left( \frac{2\nu }{\sinh
^{2}\left( c\sqrt{2\nu }\right) }\right) 
\end{eqnarray*}
and, by the second formula on page 642 in \cite{BorodinSalminen:2002} 
\begin{eqnarray*}
\mathcal{L}_{\nu }^{-1}\left( \frac{2\nu }{\sinh ^{2}\left( c\sqrt{2\nu }%
\right) }\right)  &=&4\sum_{k=0}^{\infty }\frac{\Gamma \left( 2+k\right)
e^{-\left( 2c+2kc\right) ^{2}/\left( 4t\right) }}{\sqrt{2\pi }t^{2}\Gamma
\left( 2\right) k!}D_{3}\left( \frac{2c+2kc}{\sqrt{t}}\right)  \\
&=&\frac{8c}{\sqrt{2\pi }}\sum_{k=0}^{\infty }\left( k+1\right) ^{2}\frac{%
4\left( k+1\right) ^{2}c^{2}-3t}{t^{7/2}}e^{-2\left( k+1\right) ^{2}c^{2}/t},
\end{eqnarray*}
where $D_{3}$ denotes parabolic cylinder function of order $3.$ 

Now, by (\ref
{EUTVDTV}) and the Borel convolution theorem we obtain 
\begin{eqnarray}
&\lefteqn{ \mathbb{E}\left( UTV^{c }\left( W,T\right) \cdot DTV^{c }\left( W,T\right) %
\right) }  \notag \\
&=&\mathcal{L}_{\nu }^{-1}\left( \frac{2\nu +\mu ^{2}}{2\nu ^{3}\sinh
^{2}\left( c\sqrt{2\nu +\mu ^{2}}\right) }\right) =\frac{1}{2}\int_{0}^{T}%
\frac{\left( T-t\right) ^{2}}{2}\mathcal{L}_{\nu }^{-1}\left( \frac{2\nu
+\mu ^{2}}{\sinh ^{2}\left( c\sqrt{2\nu +\mu ^{2}}\right) }\right) dt  \notag
\\
&=&\frac{2c}{\sqrt{2\pi }}\sum_{k=0}^{\infty }\left( k+1\right)
^{2}\int_{0}^{T}\left( T-t\right) ^{2}\frac{4\left( k+1\right) ^{2}c^{2}-3t}{%
t^{7/2}}e^{-\mu ^{2}t/2-2\left( k+1\right) ^{2}c^{2}/t}dt.  \label{EUTVDTV1}
\end{eqnarray}
Finally, by (\ref{EUTVDTV1}) and \cite[formula (28)]{Lochowski:2011} (notice that in \cite[formula (28)]{Lochowski:2011} and in \cite[formula (27)]{Lochowski:2011} term $\mu ^{2}t$ shall be changed into 
$\mu ^{2}t/2$) 
\begin{eqnarray}
&\lefteqn{ \text{Cov}\left( UTV^{c }\left( W,T\right), DTV^{c }\left( W,T\right) %
\right) } \notag \\
&=&\frac{2c}{\sqrt{2\pi }}\sum_{k=0}^{\infty }\left( k+1\right)
^{2}\int_{0}^{T}\left( T-t\right) ^{2}\frac{4\left( k+1\right) ^{2}c^{2}-3t}{%
t^{7/2}}e^{-\mu ^{2}t/2-2\left( k+1\right) ^{2}c^{2}/t}dt \notag \\
&&-\frac{1}{2\pi }\left( \sum_{k=0}^{\infty }\int_{0}^{T}\left( T-t\right) 
\frac{\left( 2k+1\right) ^{2}c^{2}-t}{t^{5/2}}e^{-\mu ^{2}t/2-\left(
2k+1\right) ^{2}c^{2}/\left( 2t\right) }dt\right) ^{2}. \label{covvUTVDTVT}
\end{eqnarray}

Numerical experiments show that the formula (\ref{covvUTVDTVT}) gives negative numbers, but the strict proof of this fact is not known for the author. 
\begin{rem}
Inverting formula (\ref{EUTV2}) - using second formula on page 642 in \cite{BorodinSalminen:2002} - and using just calculated covariance it is possible to obtain formula for $\text{Var}UTV^{c}\left( W,T\right)$ and hence $\text{Cor}\left( UTV^{c}\left( W,T\right), DTV^{c}\left( W,T\right)\right)$. However, numerical experiments 
show that obtained formulas are rather unstable for small $c$s. On the other hand, results of \cite{Lochowski:2010fk} give the exact value of $\text{Cor}\left( UTV^{c}\left( W,T\right), DTV^{c}\left( W,T\right)\right)$ as $c \downarrow 0,$ namely 
$$\lim_{c \downarrow 0}\text{Cor}\left( UTV^{c}\left( W,T\right), DTV^{c}\left( W,T\right)\right) = - \frac{1}{2}. $$
\end{rem}

\section*{Acknowledgments}

The author would like to express his gratitude to Prof. Przemys\l aw
Wojtaszczyk from Warsaw University for very helpful conversations which
facilitated the finding of the solutions of the two problems defined in Section 2 
and to Prof. Rimas Norvai\v{s}a from Vilnius University for
pointing out to him the remark about regulated functions. He would like also
thank Prof. Krzysztof Burdzy for encouraging him to submit this paper by
saying that the problems considered are interesting.

\end{document}